\documentclass{article}

\usepackage[utf8]{inputenc}
\usepackage[T1]{fontenc}
\usepackage{hyperref, url, booktabs, enumitem}
\usepackage{amsfonts, nicefrac, microtype, lipsum, graphicx}
\usepackage{amsmath, amssymb, amsthm}
\usepackage[all]{xy}
\SelectTips{cm}{12}

\newtheorem{lemma}{Lemma}
\newtheorem{prop}{Proposition}
\newtheorem{theorem}{Theorem}
\newtheorem*{theorem*}{Theorem}

\newcommand{\dotminus}{\mathbin{\dot{-}}}
\newcommand{\dotoplus}{\mathbin{\dot{\oplus}}}

\DeclareMathOperator\mat{M}

\DeclareMathOperator\unit{U}

\DeclareMathOperator\gunit{GU}
\DeclareMathOperator\punit{PU}
\DeclareMathOperator\glin{GL}
\DeclareMathOperator\slin{SL}
\DeclareMathOperator\pglin{PGL}
\DeclareMathOperator\gglin{GGL}
\DeclareMathOperator\symp{Sp}
\DeclareMathOperator\psymp{PSp}
\DeclareMathOperator\gsymp{GSp}
\DeclareMathOperator\orth{O}
\DeclareMathOperator\sorth{SO}
\DeclareMathOperator\psorth{PSO}
\DeclareMathOperator\gorth{GO}
\DeclareMathOperator\gsorth{GSO}
\DeclareMathOperator\spin{Spin}

\DeclareMathOperator\clif{Cl}
\DeclareMathOperator\ofalin{AL}
\DeclareMathOperator\ofasymp{ASp}
\DeclareMathOperator\ofaorth{AO}
\DeclareMathOperator\Cent{C}

\DeclareMathOperator\tr{tr}

\DeclareMathOperator\Ker{Ker}
\DeclareMathOperator\hypspace{H}
\DeclareMathOperator\sesq{Sesq}
\DeclareMathOperator\quadr{Quad}
\newcommand{\eps}{\varepsilon}

\newcommand{\inv}[1]{\!\;\overline{\!\!\:#1\vphantom !\!\!\:}\;\!}
\newcommand{\oppose}{{\mathrm{op}}}

\DeclareMathOperator{\rep}{rep}
\DeclareMathOperator{\id}{id}
\DeclareMathOperator{\Hom}{Hom}
\DeclareMathOperator{\End}{End}
\DeclareMathOperator{\Aut}{Aut}
\DeclareMathOperator{\Heis}{Heis}
\DeclareMathOperator{\Dickson}{D}
\DeclareMathOperator{\Spec}{Spec}

\newcommand{\up}[2]{{^{#1}\!{#2}}}

\makeatletter
\newcommand{\bigperp}{\mathop{\mathpalette\bigp@rp\relax}\displaylimits}
\newcommand{\bigp@rp}[2]{\vcenter{\m@th\hbox{\scalebox{\ifx#1\displaystyle2.1\else1.5\fi}{$#1\perp$}}}}
\makeatother

\newcommand{\rMod}[2][]{\mathbf{Mod}_{#1}\text{-}#2}

\title{Twisted forms of classical groups}

\author{
  Egor Voronetsky
  \thanks{Research is supported by the Russian Science Foundation grant 19-71-30002.} \\
  Chebyshev Laboratory, \\
  St. Petersburg State University, \\
  14th Line V.O., 29B, \\
  Saint Petersburg 199178 Russia \\
}

\begin{document}
\maketitle

\begin{abstract}
We give a unified description of twisted forms of classical reductive groups schemes. Such group schemes are constructed from algebraic objects of finite rank, excluding some exceptions of small rank. These objects, augmented odd form algebras, consist of $2$-step nilpotent groups with an action of the underlying commutative ring, hence we develop basic descent theory for them. In addition, we describe classical isotropic reductive groups as odd unitary groups up to an isogeny.
\end{abstract}

\section{Introduction}

The classical groups $\glin(n, K)$, $\symp(2n, K)$, and $\orth(n, K)$ are the automorphism groups of classical forms over finite free $K$-modules. Anthony Bak in \cite{Bak} gave a unified approach to all these groups and some of their variations using unitary groups. Even more general definitions are given in \cite{OddDefPetrov} by Victor Petrov and later in our work \cite{OddStrucVor}.

There are well-known constructions of all twisted forms of some of the classical groups. Twisted forms of $\glin(n, -)$ and $\symp(2n, -)$ are the unitary groups of certain Azumaya algebras with involutions. If $2$ is invertible, then the same is true for the special orthogonal group schemes $\sorth(n, -)$. For the group scheme $\sorth(2n, -)$ it is possible to add a new structure to the Azumaya algebra, a quadratic pair, to get the desired description of twisted forms. See \cite{BaptisteFasel, InvBook, DefTits, DefWeil} for details. Twisted forms of the group scheme $\sorth(2n + 1, -)$ are easily described in terms of modules with a quadratic form and a volume form, since $\Aut(\sorth(2n + 1, -)) \cong \sorth(2n + 1, -) = \orth(2n + 1, -) \cap \slin(2n + 1, -)$.

In \cite{OFADefVor} we discovered odd form algebras. These objects are similar to a module with a hermitian form and an odd form parameter in the sense of Petrov, but they are invariant under quadratic analogue of Morita equivalence (see \cite{OddStrucVor}). An odd form algebra consists of a ring $R$ and a $2$-step nilpotent group $\Delta$ (odd form parameter) with various operations satisfying a lot of axioms. It is possible to construct odd form algebras over a commutative ring $K$ with the unitary groups $\glin(n, K)$, $\symp(2n, K)$, and $\orth(n, K)$, their rings are the usual Azumaya algebras with involutions (in the odd orthogonal case only if $2 \in K^*$).

These odd form algebras do not naturally form fppf-sheaves, we even cannot define useful scalar extension for odd form parameters in general. Hence we prefer to work with odd form algebras were the odd form parameter has a fixed nilpotent filtration $0 \leq \mathcal D \leq \Delta$ satisfying a couple of axioms. For these objects, augmented odd form algebras, we develop faithfully flat descent in the required generality.

Roughly speaking, our main result is the following: for every twisted form of the group schemes $\glin(n, -)$, $\slin(n, -)^m$, $\symp(2n, -)^m$, $\sorth(n, -)^m$, $\pglin(n, -)^m$, $\psymp(2n, -)^m$, $\psorth(n, -)^m$, $\gsymp(2n, -)$, $\gsorth(n, -)$, and $\spin(n, -)^m$ there is a construction using certain augmented odd form $K$-algebras. Moreover, there is a nice description of isotropic rank in terms of augmented odd form algebras, see \cite{PetrovStavrova, Stavrova} for a definition of isotropic reductive groups.

The only exceptions to our result either have non-affine automorphism group scheme, or are twisted forms of $\spin(8, -)^m$ or $\psorth(8, -)^m$. In the second case we do not classify all the twisted forms because of the triality, though the twisted forms with the isotropic rank at least $3$ still can be described.

The paper is organized as follows. In sections $2$ and $3$ we give the necessary definitions of quadratic modules and odd form algebras. In section $4$ we discuss descent for $2$-step nilpotent groups with suitable actions of a base commutative ring $K$ (the so-called $2$-step nilpotent $K$-modules). Section $5$ contains the constructions of the split classical odd form algebras. In sections $6$--$8$ we prove the main results.

\section{Quadratic forms}

All rings in this paper are associative, but not necessarily unital. All commutative rings are unital and homomorphisms between them preserve the identity, similarly for commutative algebras over a given commutative ring. We also consider only unital modules and bimodules over unital associative rings. The center of a ring $R$ is denoted by $\Cent(R)$, similarly for the center of a group. For any associative ring $R$ its multiplicative semigroup is denoted by $R^\bullet$ (it is a monoid if $R$ is unital). If $R$ is unital, then $R^*$ denotes its group of invertible elements.

If $R$ is a non-unital $K$-algebra for a commutative ring $K$, then $R \rtimes K$ is a unital $K$-algebra with an ideal $R$. Here $R \rtimes K = R \oplus K$ as a $K$-module and the multiplication is given by $(r \oplus k) (r' \oplus k') = (rr' + rk' + r'k) \oplus kk'$.

We use the conventions $\up gh = ghg^{-1}$ and $[g, h] = ghg^{-1}h^{-1}$, where $g$ and $h$ are elements of some group. In nilpotent groups we usually write the group operation as $\dotplus$, then $\dotminus u$ is the inverse, $\dot 0$ is the neutral element, $u \dotminus v = u \dotplus (\dotminus v)$.

If $\{A_i\}_{i \in I}$ is a linearly ordered family of subsets of a nilpotent group and $\dot 0 \in A_i$ for all $i$, then
$$\sum^\cdot_{i \in I} A_i = \bigl\{\sum_{i \in I}^\cdot a_i \mid a_i \in A_i, \text{almost all } a_i = \dot 0\bigr\}.$$
If, moreover, the summands $a_i$ are uniquely determined by their sum, then we write $\bigoplus_{i \in I}^\cdot A_i$ instead of $\sum_{i \in I}^\cdot A_i$. Note that $\bigoplus_{i \in I}^\cdot A_i$ makes sense for any linearly ordered family of pointed sets $A_i$, it is the set of finite formal sums.

For every ring $R$ the opposite ring is $R^\oppose = \{r^\oppose \mid r \in R\}$ with operations $r^\oppose + {r'}^\oppose = (r + r')^\oppose$ and $r^\oppose {r'}^\oppose = (r' r)^\oppose$. If $M_R$ is a right module over a unital ring $R$, then the opposite left $R^\oppose$-module is $M^\oppose = \{m^\oppose \mid m \in M\}$ with $m^\oppose + {m'}^\oppose = (m + m')^\oppose$ and $r^\oppose m^\oppose = (mr)^\oppose$. We define the opposite to a left module in the same way. There is a canonical isomorphism $(R^\oppose)^\oppose \cong R, (r^\oppose)^\oppose \mapsto r$ for every ring $R$ and similarly for modules.

We assume that the reader is familiar with reductive group schemes. The necessary results are given with references in SGA3 \cite{SGA3} and Brian Conrad's book \cite{Conrad}. All schemes over a given commutative ring $K$ are identified with the corresponding covariant functors on the category of commutative $K$-algebras.

Now we give slightly more natural definitions in comparison with \cite{OddStrucVor}.

A quadratic ring is a triple $(R, L, A)$, where $R$ is a unital ring, $L$ is an $R^\oppose$-$R$-bimodule, $A$ is a right $R^\bullet$-module (the action is denoted by $a \cdot r$), and there are additive maps $L \to L, l \mapsto \inv{\,l\,}$, $\varphi \colon L \to A$, and $\tr \colon A \to L$ such that
\begin{itemize}
\item $\inv{\inv{\,l\,}} = l$, $\inv{r^\oppose l r'} = {r'}^\oppose \inv{\,l\,} r$;
\item $\varphi(r^\oppose l r) = \varphi(l) \cdot r$, $\tr(a \cdot r) = r^\oppose \tr(a) r$;
\item $\tr(\varphi(l)) = l + \inv{\,l\,}$;
\item $\tr(a) = \inv{\tr(a)}$, $\varphi(l) = \varphi(\inv{\,l\,})$;
\item $a \cdot (r + r') = a \cdot r + \varphi(r^\oppose \tr(a) r') + a \cdot r'$.
\end{itemize}

In \cite{OddStrucVor} the ring $R$ itself has an anti-automorphism $r \mapsto r^*$ such that $r^{**} = \lambda r \lambda^{-1}$ for some fixed invertible element $\lambda$. The map $r \mapsto r^*$ is called a pseudo-involution with symmetry $\lambda$. In this case take $L = R$ as the $R^\oppose$-$R$-bimodule with $r_1^\oppose r_2 r_3 = r_1^* r_2 r_3$ and $\inv r = r^* \lambda$. The conditions on $A$, $\varphi$, and $\tr$ from the definition are precisely the conditions on quadratic structures from that paper.

Let $(R, L, A)$ be a quadratic ring. We say that $(M, B, q)$ is a quadratic module over $(R, L, A)$, if $M$ is a right $R$-module and $B \colon M \times M \to L$, $q \colon M \to A$ are maps such that
\begin{itemize}
\item $B$ is biadditive, $B(mr, m'r') = r^\oppose B(m, m') r'$ (i.e. $B$ induces a bimodule homomorphism $M^\oppose \otimes M \to R$);
\item $B(m, m') = \inv{B(m', m)}$;
\item $\tr(q(m)) = B(m, m)$, $q(mr) = q(m) \cdot r$;
\item $q(m + m') = q(m) + \varphi(B(m, m')) + q(m')$.
\end{itemize}
Here $B$ is called a hermitian form and $q$ is called a quadratic form.

In Bak's definition \cite{Bak} hermitian forms take values in the ring $R$ itself (using a pseudo-involution with symmetry) and quadratic forms take values in factor-groups of $R$ by so-called form parameters. It is possible to define form parameters in general following O. Loos \cite{Loos}. A quadratic ring $(R, L, A)$ is called even if $\varphi \colon L \to A$ is surjective, in this case $\Lambda = \Ker(\varphi)$ is called a form parameter. It is an $R^\bullet$-submodule of $L$ such that $\Lambda_{\mathrm{min}} \leq \Lambda \leq \Lambda_{\mathrm{max}}$, where $\Lambda_{\mathrm{min}} = \{l - \inv{\,l\,} \mid l \in L\}$ and $\Lambda_{\mathrm{max}} = \{l \in L \mid l + \inv{\,l\,} = 0\}$ (where $R^\bullet$ acts by $l \cdot r = r^\oppose l r$). Conversely, for every such $\Lambda$ the triple $(R, L, L / \Lambda)$ is an even quadratic ring. In even quadratic rings there is an identity $\varphi(\tr(a)) = 2a$ for all $a \in A$.

Now we give an alternative description of quadratic forms following Petrov's paper \cite{OddDefPetrov}. Let $(M, B, q)$ be a quadratic module over $(R, L, A)$. A Heisenberg group of $(M, B)$ is the set $\Heis(M) = M \times L$ with the group operation
$$(m, l) \dotplus (m', l') = (m + m', l - B(m, m') + l')$$
and the right $R^\bullet$-action
$$(m, l) \cdot r = (mr, r^\oppose l r).$$
Note that $\dot 0 = (0, 0)$ is the neutral element and $\dotminus (m, l) = (-m, -B(m, m) - l)$. An odd form parameter of $q$ is the subgroup
$$\mathcal L = \{(m, l) \mid q(m) + \varphi(l) = 0\} \leq \Heis(M).$$
Clearly, the odd form parameter $\mathcal L$ of a quadratic module $(M, B, q)$ over a quadratic ring $(R, L, A)$ is $R^\bullet$-invariant and $\mathcal L_{\mathrm{min}} \leq \mathcal L \leq \mathcal L_{\mathrm{max}}$, where
\begin{align*}
\mathcal L_{\mathrm{min}} &= \{(0, l - \inv{\,l\,}) \mid l \in L\},\\
\mathcal L_{\mathrm{max}} &= \{(m, l) \mid B(m, m) + l + \inv{\,l\,} = 0\}.
\end{align*}
The definitions of $\Heis(M)$, $\mathcal L_{\mathrm{min}}$, $\mathcal L_{\mathrm{max}}$ do not use $A$ and $q$. If $\mathcal L_{\mathrm{min}} \leq \mathcal L \leq \mathcal L_{\mathrm{max}}$ is arbitrary $R^\bullet$-invariant subgroup, then $(R, L, \Heis(M) / \mathcal L)$ is a quadratic ring with $\varphi(l) = (0, l) \dotplus \mathcal L$ and $\tr((m, l) \dotplus \mathcal L) = B(m, m) + l + \inv{\,l\,}$, also $\mathcal L$ is the odd form parameter of the quadratic form $q(m) = (m, 0) \dotplus \mathcal L$. If $(R, L, A)$ is even, then there is a short exact sequence $\dot 0 \to \Lambda \to \mathcal L \to M \to \dot 0$.

A quadratic ring $(R, L, A)$ is called regular (or non-degenerate), if $L$ is an invertible bimodule in the sense of Morita theory, i.e. $L_R$ is a projective generator in $\rMod R$ and $R^\oppose \cong \End_R(L_R)$. A quadratic module $(M, B, q)$ over a regular quadratic ring $(R, L, A)$ is called regular (or a quadratic space), if $M$ is finite projective $R$-module and $B$ induces an isomorphism $M^\oppose \to \Hom_R(M, L)$. For example, if $(R, L, A)$ is a regular quadratic ring and $P_R$ is a finite projective module, then the hyperbolic space $\hypspace(P) = \Hom_R(P, L)^\oppose \oplus P$ is a regular quadratic module over $(R, L, A)$ with the hermitian form
$$B(f^\oppose \oplus p, {f'}^\oppose \oplus p') = f(p') + \inv{f'(p)}$$
and the quadratic form
$$q(f^\oppose \oplus p) = \varphi(f(p)).$$

Let $(R, L, A)$ be a quadratic ring. For any right $R$-module $M$ let $\sesq(M)$ be the group of all biadditive maps $Q \colon M \times M \to L$ such that $Q(mr, m'r') = r^\oppose Q(m, m') r'$, and $\quadr(M)$ be the group of pairs $(B, q)$ such that $(M, B, q)$ is a quadratic module over $(R, L, A)$. The category of right $R$-modules $\rMod R$ is an additive form category in the sense of Marco Schlichting, see \cite{Schlichting}. The duality is $M^\sharp = \Hom_R(M, L)^\oppose$, the natural transformation is $\mathrm{can}_M \colon M \to \Hom_R(\Hom_R(M, L)^\oppose, L)^\oppose, m \mapsto (f^\oppose \mapsto \inv{f(m)})^\oppose$, then $\Hom_R(M, M^\sharp) \cong \sesq(M)$ and $\sigma \colon f \mapsto f^\sharp \circ \mathrm{can}_M$ corresponds to the involution on $\sesq(M)$. The functor $Q$ is $Q(M) = \quadr(M)$, the transfer map is $\varphi \colon \sesq(M) \to \quadr(M)$, and the restriction map is $\tr \colon \quadr(M) \to \sesq(M)$. Schlichting's symmetric and quadratic forms correspond to our hermitian and quadratic forms. If $(R, L, A)$ is regular, then the subcategory of finite projective $R$-modules has strong duality, Schlichting's non-degenerate quadratic forms correspond to our regular quadratic forms, and Schlichting's hyperbolic spaces correspond to the our ones.

A unitary group of a quadratic module $(M, B, q)$ over a quadratic ring $(R, L, A)$ is
$$\unit(M, B, q) = \{g \in \Aut_R(M) \mid B(gm, l, gm') = B(m, l, m'), q(a, gm) = q(a, m)\}.$$
In other words, $\unit(M, B, q)$ is the group of automorphisms of $(M, B, q)$ as a quadratic module.

An even quadratic ring $(R, L, A)$ is called an even quadratic algebra over a commutative ring $K$ if $R$ is a $K$-algebra, both $K$-module structures on $L$ coincide, the involution on $L$ is $K$-linear, and the form parameter $\Lambda = \Ker(\varphi)$ is a $K$-submodule. In this case $A$ itself is a $K$-module. If $(R, L, A)$ is an even quadratic $K$-algebra and $E / K$ is a commutative ring extension, then $(E \otimes_K R, E \otimes_K L, E \otimes_K A)$ is an even quadratic $E$-algebra. If $(M, B, q)$ is a quadratic module over $(R, L, A)$, then $E \otimes_K M$ is a quadratic module over $(E \otimes_K R, E \otimes_K L, E \otimes_K A)$ with the obvious hermitian form $B_E$ and the quadratic form given by $q_E(e \otimes m) = e^2 \otimes q(m)$.

Now we describe classical hermitian and quadratic forms in our notation. Fix a commutative ring $K$. By the above, even quadratic algebras over $K$ and quadratic modules over them give fppf sheaves on the category of affine schemes over $K$ using extension of scalars.

On the $K$-algebra $R = K \times K$ there is an involution $\inv{(x, y)} = (x, y)^* = (y, x)$ with the symmetry $\lambda = 1$. We say that the regular even quadratic $K$-algebra $(R, R, K)$ with $x \cdot (y, z) = yxz$, $\varphi(x, y) = x + y$, and $\tr(x) = (x, x)$ is classical of linear type. Any regular quadratic module over $(R, R, K)$ is isomorphic to $\hypspace(P \times 0)$ for some finite projective $K$-module $P$, the unitary group of such a quadratic module is isomorphic to $\glin(P)$. Zariski locally every finite projective module splits, i.e. is isomorphic to $K^\ell$ for some $\ell \geq 0$, and its hyperbolic space is $\bigoplus_{i = -\ell}^{-1} K e_i \times \bigoplus_{i = 1}^\ell K e_i$. The hermitian form is given by $B(e_i, e_j) = 0$ for $i \neq -j$, $B(e_i, e_{-i}) = (1, 0)$ for $i > 0$, $B(e_i, e_{-i}) = (0, 1)$ for $i < 0$, and the quadratic form is given by $q(e_i) = 0$.

The regular even quadratic $K$-algebra $(K, K, 0)$ with $\inv k = -k$ is called classical of symplectic type. Quadratic modules over $(K, K, 0)$ are called symplectic $K$-modules. In other words, a symplectic $K$-module $M$ has a bilinear form $B \colon M \times M \to K$ satisfying $B(m, m) = 0$. Locally in the Zariski topology every regular symplectic module splits, i.e. it is isomorphic to $\bigoplus_{1 \leq |i| \leq \ell} K e_i$ with the form $B(e_i, e_{-i}) = \eps_i$, $B(e_i, e_j) = 0$ for $i \neq -j$ (here $\eps_i = 1$ for $i > 0$ and $\eps_i = -1$ for $i < 0$). Any split symplectic module is hyperbolic.

The regular even quadratic $K$-algebra $(K, K, K)$ with $\inv k = k$, $k \cdot x = kx^2$, $\varphi(k) = k$, $\tr(k) = 2k$ is called classical of orthogonal type. Quadratic modules over $(K, K, K)$ are called classical quadratic $K$-modules. In other words, a classical quadratic $K$-module is a $K$-module $M$ with a map $q \colon M \to K$ such that $q(mk) = q(m) k^2$ and $B(m, m') = q(m + m') - q(m) - q(m')$ is $K$-bilinear. Locally in the \'etale topology every regular classical quadratic module of even rank splits, i.e. is isomorphic to $\bigoplus_{1 \leq |i| \leq \ell} K e_i$ with $q(e_i) = 0$, $B(e_i, e_{-i}) = 1$, $B(e_i, e_j) = 0$ for $i \neq -j$. Such a quadratic module is hyperbolic.

If $(P, q)$ is a classical quadratic module with $P = K^{2\ell + 1}$, then there is a universal polynomial $\mathrm{hdet}(q)$ with integer coefficients on the variables $q(e_i)$ and $B(e_i, e_j)$ for $i < j$ such that $2\, \mathrm{hdet}(q) = \det B(e_i, e_j)$. See \cite[section IV.3]{Knus} or \cite[Appendix C]{Conrad} for details. The form $q$ is called semi-regular if $\mathrm{hdet}(q)$ is invertible, this property does not depend on the basis of $P$. A classical quadratic module $(P, q)$ is called semi-regular if $P$ is finite projective of odd rank and $q$ is semi-regular Zariski locally. Locally in the fppf topology every semi-regular classical quadratic module splits, i.e. is isomorphic to $\bigoplus_{i = -\ell}^\ell K e_i$ with $q(e_i) = 0$ for $i \neq 0$, $q(e_0) = 1$, $B(e_i, e_{-i}) = 1$ for $i \neq 0$, $B(e_i, e_j) = 0$ for $i \neq -j$, $B(e_0, e_0) = 2$. Such a quadratic module is an orthogonal sum of a submodule of rank $1$ and a hyperbolic subspace.

\section{Odd form rings}

Recall the definitions from \cite{UnitK2, OFADefVor}. An odd form ring is a pair $(R, \Delta)$, where $R$ is a non-unital associative ring with an involution $a \mapsto \inv a$ (i.e. an anti-automorphism of order $2$), $\Delta$ is a group, the semi-group $R^\bullet$ acts on $\Delta$ from the right by endomorphisms (the group operation of $\Delta$ is denoted by $\dotplus$ and the action is denoted by $u \cdot a$), and there are maps $\phi \colon R \to \Delta$, $\pi \colon \Delta \to R$, and $\rho \colon \Delta \to R$ such that for all $u, v \in \Delta$ and $a, b \in R$
\begin{itemize}
\item $\pi(u \dotplus v) = \pi(u) + \pi(v)$, $\pi(u \cdot a) = \pi(u) a$;
\item $\phi(a + b) = \phi(a) \dotplus \phi(b)$, $\phi(b) \cdot a = \phi(\inv aba)$;
\item $\rho(u \dotplus v) = \rho(u) - \inv{\pi(u)} \pi(v) + \rho(v)$, $\rho(u \cdot a) = \inv a \rho(u) a$;
\item $\rho(u) + \inv{\rho(u)} + \inv{\pi(u)} \pi(u) = 0$;
\item $\pi(\phi(a)) = 0$, $\rho(\phi(a)) = a - \inv a$;
\item $[u, v] = \phi(-\inv{\pi(u)} \pi(v))$;
\item $\phi(a) = \dot 0$ if $a = \inv a$;
\item $u \cdot (a + b) = u \cdot a \dotplus \phi(\inv{\,b\,} \rho(u) a) \dotplus u \cdot b$.
\end{itemize}
An odd form ring $(R, \Delta)$ is called unital if $R$ is unital and $u \cdot 1 = u$ for all $u \in \Delta$. An odd form ring $(R, \Delta)$ is called special if $(\pi, \rho) \colon \Delta \to R \times R$ is injective.

If $(R, \Delta)$ is an odd form ring, then $\Delta$ is a $2$-step nilpotent group. There are identities $\rho(\dot 0) = 0$, $\rho(\dotminus u) = \inv{\rho(u)}$, and $u \cdot 0 = \dot 0$ for $u \in \Delta$. If $R$ is a unital ring with an involution and $\Delta \leq \Heis(R)$ is an odd form parameter with respect to the hermitian form $B(a, b) = \inv ab$, then $(R, \Delta)$ is a special unital odd form ring, where $\pi$ and $\rho$ are the first and the second projections of the Heisenberg group, and $\phi(a) = (0, a - \inv a) \in \Delta$. Conversely, every special unital odd form ring $(R, \Delta)$ is of this type up to an identification of $\Delta$ with its image in $R \times R$.

Let $(R, \Delta)$ be an odd form ring. It is an odd form subring of the unital odd form ring $(R \rtimes \mathbb Z, \Delta)$ (and an odd form ideal as in \cite{UnitK2, OFADefVor}). In order to define its unitary group we denote the components of $g \in R \times \Delta$ by $\beta(g)$ and $\gamma(g)$ as in \cite{UnitK2, OFADefVor}. Also, $\alpha(g) = \beta(g) + 1 \in R \rtimes \mathbb Z$. Now a unitary group $\unit(R, \Delta)$ consists of elements $g \in R \times \Delta$ such that with $\alpha(g)^{-1} = \inv{\alpha(g)}$, $\pi(\gamma(g)) = \beta(g)$, and $\rho(\gamma(g)) = \inv{\beta(g)}$. The group operation is given by $\alpha(gg') = \alpha(g)\, \alpha(g')$ and $\gamma(gg') = \gamma(g) \cdot \alpha(g') \dotplus \gamma(g')$. In the special unital case $\unit(R, \Delta) \cong \{g \in R^* \mid g^{-1} = \inv g, (g - 1, \inv g - 1) \in \Delta\}$.

Recall a construction of an odd form ring by a quadratic module from \cite{OFADefVor}. Let $(R, L, A)$ be a quadratic ring, $(M, B, q)$ be a quadratic module over it. Consider a special unital odd form ring $(T', \Xi')$ with $T' = \End_R(M)^\oppose \times \End_R(M)$, the involution $\inv{(a^\oppose, b)} = (b^\oppose, a)$, and $\Xi' = \Xi'_{\mathrm{max}}$ be the maximal odd form parameter with respect to the hermitian form $B_{T'}(t, t') = \inv{\,t\,} t'$. Here $\unit(T', \Xi') \cong \Aut_R(M)$ is the group of module automorphisms of $M$. Now let
\begin{align*}
T &= \{(x^\oppose, y) \in T' \mid B(m, ym') = B(xm, m')\},\\
\Xi &= \{(x^\oppose, y; z^\oppose, w) \in \Xi_{\mathrm{max}} \mid q(ym) + \varphi(B(m, wm)) = 0\}.
\end{align*}
Then $(T, \Xi) \subseteq (T', \Xi')$ is a special unital odd form subring. Moreover, under the isomorphism $\unit(T', \Xi') \cong \Aut_R(M)$ the group $\unit(T, \Xi)$ corresponds to $\unit(M, B, q)$. We say that $(T, \Xi)$ is obtained from $(M, B, q)$ by a naive construction.

Now we give another construction of an odd form ring by a quadratic module.

Let $(R, L, A)$ be a regular quadratic ring and $(M, B, q)$ be a quadratic module over it. A canonical construction of an odd form ring $(S, \Theta)$ from $(M, B, q)$ goes as follows. Recall that $L$ induces a Morita equivalence between $R$ and $R^\oppose$. Hence there is an $R$-$R^\oppose$-bimodule $L^{-1}$ and coherent pairing isomorphisms $L \otimes_R L^{-1} \to R^\oppose$, $L^{-1} \otimes_{R^\oppose} L \to R$. For simplicity, we consider these isomorphisms as identity maps and do not write the symbol $\otimes$ for elements of any tensor products involving $L$ or $L^{-1}$. The involution on $L$ induces a unique involution on $L^{-1}$ such that $\inv{\,l\,} \inv{\,t\,} = (tl)^\oppose \in R^\oppose$ and $(\inv{\,t\,} \inv{\,l\,})^\oppose = lt \in R^\oppose$ for all $l \in L$ and $t \in L^{-1}$. Let
$$S = M \otimes_R L^{-1} \otimes_{R^\oppose} M^\oppose,$$
it is a non-unital involution ring with the operations
\begin{align*}
(m_1 t_1 {m'_1}^\oppose) (m_2 t_2 {m'_2}^\oppose) &= m_1 t_1\, B(m'_1, m_2)\, t_2 {m'_2}^\oppose,\\
\inv{m t {m'}^\oppose} &= m' \inv{\,t\,} m^\oppose.
\end{align*}
In order to define $\Theta$ let $N = L^{-1} \otimes_{R^\oppose} M^\oppose$, it is a left $R$-module and a right non-unital $S$-module. For the odd form parameter $\mathcal L$ of $(M, B, q)$ we denote the projections $\mathcal L \to M$ and $\mathcal L \to L$ by $\pi$ and $\rho$, as for odd form rings. As an abstract group $\Theta$ is generated by elements $\phi(s)$ for $s \in S$ and $u \boxtimes n$ for $u \in \mathcal L$, $n \in N$ with the relations
\begin{itemize}
\item $\phi(s + s') = \phi(s) \dotplus \phi(s')$;
\item $[\phi(s), u \boxtimes n] = \dot 0$;
\item $[u \boxtimes n, u' \boxtimes n'] = \phi\bigl(-n^\oppose\, B(\pi(u), \pi(u'))\, n'\bigr)$;
\item $(u \dotplus u') \boxtimes n = u \boxtimes n \dotplus u' \boxtimes n$;
\item $(u \cdot r) \boxtimes n = u \boxtimes rn$;
\item $u \boxtimes (n + n') = u \boxtimes n \dotplus \phi({n'}^\oppose\, \rho(u)\, n) \dotplus u \boxtimes n'$;
\item $(0, l - \inv{\,l\,}) \boxtimes n = \phi(n^\oppose l n)$;
\item $\phi(s) = \dot 0$ if $s = \inv s$.
\end{itemize}
The operations are by
\begin{itemize}
\item $\pi(\phi(s)) = 0$, $\pi(u \boxtimes n) = \pi(u)\, n$;
\item $\rho(\phi(s)) = s - \inv s$, $\rho(u \boxtimes n) = n^\oppose\, \rho(u)\, n$;
\item $\phi(s') \cdot s = \phi(\inv s s' s)$, $(u \boxtimes n) \cdot s = u \boxtimes ns$.
\end{itemize}

\begin{lemma}\label{canon-well-def}
The canonical construction gives a well-defined odd form ring $(S, \Theta)$. If $N = \bigoplus_{i \in I} R e_i$ is a free left $R$-module, then $\Theta = \bigoplus_i^\cdot (\mathcal L \boxtimes e_i) \dotoplus \bigoplus_{i < j}^\cdot \phi(e_i^\oppose L e_j)$ for any linear order on $I$. If $M$ is flat over $R$, then $(S, \Theta)$ is special.
\end{lemma}
\begin{proof}
It is easy to see that $\pi \colon \Theta \to S$ and $\phi \colon S \to \Theta$ are well-defined group homomorphisms. The map $\rho \colon \Theta \to S$ is a well-defined quadratic map in the sense of M. Hartl by \cite[proposition 1.20]{Hartl}: the cross effect $-\inv{\pi(u)} \pi(v)$ is well-defined as a biadditive map, it takes values in the center of the target group $S$, and the values of $\rho$ on both sides of each relation (evaluated using the cross effect) coincide. By \cite[corollary 1.19]{Hartl} these operations satisfy all axioms not involving the action of $S^\bullet$ on $\Theta$. The second claim is clear, the odd form ring $(S, \Theta)$ is actually special if $N$ is a free $R$-module.

The action map $\Theta \times S \to \Theta$ may be considered as a homomorphism from $\Theta$ to the group of quadratic maps $S \to \Theta$ with cross effects taking values in $\phi(S)$, where the group operation is the pointwise addition. Such a homomorphism is clearly well-defined. The remaining axioms follows from a direct calculation.

In order to prove the last claim recall Govorov -- Lazard theorem: the flat left $R$-module $N$ is a direct limit of finite free modules $Q_i$. Let $P_i = Q_i^\oppose \otimes_{R^\oppose} L$, so $Q_i = L^{-1} \otimes_{R^\oppose} P_i^\oppose$. Consider the odd form rings $(S_i, \Theta_i)$ constructed from the quadratic modules $(P_i, B|_{P_i \times P_i}, q|_{P_i})$ over $(R, L, A)$. By the second claim they are special and by construction $(S, \Theta) = \varinjlim_i (S_i, \Theta_i)$. A direct limit of special odd form rings is clearly special.
\end{proof}

Now we have odd form rings $(T, \Xi)$ and $(S, \Theta)$ constructed from a quadratic module $(M, B, q)$ over a regular quadratic ring $(R, L, A)$. There is a morphism $f \colon (S, \Theta) \to (T, \Xi)$ such that
$$f(m t {m'}^\oppose) = \bigl((m' \inv{\,t\,} B(m, -))^\oppose, m t B(m', -)\bigr) \in T$$
for $m, m' \in M$ and $t \in L^{-1}$.

\begin{prop}\label{naive-canon}
Let $(R, L, A)$ be a regular quadratic ring and $(M, B, q)$ be a regular quadratic module over it. Then the canonical morphism $f \colon (S, \Theta) \to (T, \Xi)$ from the canonical odd form ring of $(M, B, q)$ to the naive one is an isomorphism.
\end{prop}
\begin{proof}
It is easy to see that $f$ gives an isomorphism between $S$ and $T$, also $T \cong \End_R(M)$. If $(x^\oppose, y; z^\oppose, w) \in \Xi$, then $(ym, B(m, wm)) \in \mathcal L$ for all $m \in M$, hence $(x^\oppose, y; z^\oppose, w)$ is the image of
$$
\sum_i^\cdot \bigl((ym_i, B(m_i, wm_i)) \boxtimes t_i {m'_i}^\oppose\bigr) \dotplus \phi\bigl(\sum_{i < j} m'_j \inv{t_j}\, B(m_j, wm_i)\, t_i {m'_i}^\oppose\bigr)
$$
for some decomposition $1 = \sum_i m_i t_i {m'_i}^\oppose \in S$. In other words, $f \colon \Theta \to \Xi$ is surjective. It is injective, because the odd form rings are special.
\end{proof}

Now let $K$ be a commutative ring. An odd form ring $(R, \Delta)$ is called an odd form $K$-algebra if $R$ is a $K$-algebra, the involution on $R$ is $K$-linear, and there is a unital right action of the monoid $K^\bullet$ on $\Delta$ by endomorphisms such that $(R \rtimes K, \Delta)$ is a unital odd form ring. There is an equivalent definition using a list of identities, see \cite{UnitK2}. Every odd form ring is an odd form algebra over $\mathbb Z$ in a unique way.

If $(R, L, A)$ is an even quadratic $K$-algebra and $(M, B, q)$ is a quadratic module over it, then odd form ring $(T, \Xi)$ obtained by the naive construction is an odd form $K$-algebra. Moreover, if $(R, L, A)$ is regular, then the odd form ring $(S, \Theta)$ obtained by the canonical construction is also an odd form $K$-algebra and the morphism $(S, \Theta) \to (T, \Xi)$ preserves this additional structure. The action of $K^\bullet$ on $\Theta$ is given by
$$
\phi(s) \cdot k = \phi(k^2 s), \quad
(u \boxtimes n) \cdot k = u \boxtimes nk.
$$

Recall a definition from \cite{UnitK2}. An augmented odd form $K$-algebra is a triple $(R, \Delta, \mathcal D)$ such that $(R, \Delta)$ is an odd form $K$-algebra, $\mathcal D \leq \Delta$ is an $R^\bullet$-invariant subgroup and a left $K$-module, and for all $a \in R$, $k \in K$, $v \in \mathcal D$
\begin{itemize}
\item $\phi(a) \in \mathcal D$, $\phi(ka) = k \phi(a)$;
\item $\pi(v) = 0$, $v \cdot k = k^2 v$, $\rho(kv) = k \rho(v)$, $kv \cdot a = k(v \cdot a)$.
\end{itemize}
Any odd form $K$-algebra $(R, \Delta)$ may be considered as an augmented odd form $K$-algebra $(R, \Delta, \phi(R))$.

\begin{lemma}\label{canon-aug}
Let $(R, L, A)$ be a regular even quadratic $K$-algebra and $(M, B, q)$ be a quadratic module over it. Suppose that $M$ is flat over $R$. Then the canonical odd form $K$-algebra $(S, \Theta)$ has an augmentation $\mathcal F \leq \Theta$ generated by the elements $\phi(s)$ for $s \in S$ and $u \boxtimes n$ for $u \in \Lambda \leq \mathcal L$ and $n \in N$. The operations are given by
$$k \phi(s) = \phi(ks), \quad k (u \boxtimes n) = ku \boxtimes n.$$
Moreover, $\mathcal F = \Ker(\pi)$ and as an abstract group it has the same presentation as $\Theta$ but with $u, u' \in \Lambda$.
\end{lemma}
\begin{proof}
As in the proof of lemma \ref{canon-well-def}, we may assume that $N = \bigoplus_{i \in I} Re_i$ is free over $R$ by Govorov -- Lazard theorem. In this case $\mathcal F = \bigoplus_i^\cdot (\Lambda \boxtimes e_i) \dotoplus \bigoplus_{i < j}^\cdot \phi(e_i^\oppose L e_j)$ for any linear order on $I$ and all claims easily follow.
\end{proof}

Augmented odd form algebras with flat factor-groups $\Delta / \mathcal D$ over $K$ are precisely the objects we need to describe all twisted forms of classical groups.

\section{Nilpotent descent}

From now on we fix a commutative ring $K$.

Recall from \cite{UnitK2} that a $2$-step nilpotent $K$-module is a pair $(M, M_0)$, where $M$ is a group with a unital right $K^\bullet$-action (the group operation is denoted by $\dotplus$ and action is denoted by $m \cdot k$), $M_0$ is a subgroup of $M$ and a left $K$-module, and there is a map $\tau \colon M \to M_0$ such that
\begin{itemize}
\item $[M, M] \leq M_0$, $[M, M_0] = \dot 0$;
\item $[m \cdot k, m' \cdot k'] = kk' [m, m']$;
\item $m \cdot (k + k') = m \cdot k \dotplus kk' \tau(m) \dotplus m \cdot k'$;
\item $m \cdot k = k^2 m$ for $m \in M_0$.
\end{itemize}
In a $2$-step nilpotent $K$-module $(M, M_0)$ there are identities $\tau(m) = 2m$ for $m \in M_0$, $\tau(m \cdot k) = k^2 \tau(m)$ for $m \in M$ and $k \in K$, and $\tau(m \dotplus m') = \tau(m) + [m, m'] + \tau(m')$ for $m, m' \in M$. Also, $M / M_0$ has a natural structure of a right $K$-module. For example, if $(R, \Delta, \mathcal D)$ is an augmented odd form $K$-algebra, then $(\Delta, \mathcal D)$ is a $2$-step nilpotent $K$-module with $\tau(u) = u \dotplus u \cdot (-1) = \phi(\rho(u))$. 

Morphisms of $2$-step nilpotent $K$-modules $f \colon (M, M_0) \to (N, N_0)$ are the group homomorphisms $f \colon M \to N$ mapping $M_0$ to $N_0$ and preserving all operations. Such a morphism induces $K$-module homomorphisms $f_0 \colon M_0 \to N_0$ and $f_1 \colon M / M_0 \to N / N_0$. If both $f_0$ and $f_1$ are injective, then $f$ is also injective by $5$-lemma for groups, and similarly for surjectivity.

For example, let $M_0$ and $M_1$ be left $K$-modules, $b \colon M_1 \times M_1 \to M_0$ be a bilinear map. Then $b$ is a $2$-cocycle in the sense of group theory, the central extension $M = M_1 \dotoplus M_0$ is a group with $(m_1 \dotoplus m_0) \dotplus (m_1' \dotoplus m_0') = (m_1 + m_1') \dotoplus (m_0 + b(m_1, m'_1) + m_0')$, and $M_0$ is naturally a subgroup of $M$ with the factor-group isomorphic to $M_1$. Moreover, $(M, M_0)$ is a $2$-step nilpotent $K$-module with $(m_1 \dotoplus m_0) \cdot k = km_1 \dotoplus k^2 m_0$ and $\tau(m_1 \dotoplus m_0) = 2m_0 - b(m_1, m_1)$.

A $2$-step nilpotent $K$-module $(M, M_0)$ is called split if it is isomorphic to such a central extension. It is easy to see that every $2$-step nilpotent $K$-module with free $M / M_0$ over $K$ splits.

Now we define extension of scalars. Let $(M, M_0)$ be a $2$-step $K$-module and $K \to E$ be a homomorphism of commutative rings. A scalar extension is the pair $(M, M_0) \boxtimes_K E = (M \boxtimes_K E, \mathrm{Im}(E \otimes_K M_0))$, where the group $M \boxtimes_K E$ is the abstract group generated by $E \otimes_K M_0$ and the elements $m \boxtimes e$ for $m \in M$, $e \in E$ with the relations
\begin{itemize}
\item $[m \boxtimes e, x] = \dot 0$ for $x \in E \otimes_K M_0$;
\item $[m \boxtimes e, m' \boxtimes e'] = ee' \otimes [m, m']$;
\item $(m \dotplus m') \boxtimes e = m \boxtimes e \dotplus m' \boxtimes e$;
\item $(m \cdot k) \boxtimes e = m \boxtimes ke$;
\item $m \boxtimes (e + e') = m \boxtimes e \dotplus ee' \otimes \tau(m) \dotplus m \boxtimes e'$;
\item $m \boxtimes e = e^2 \otimes m$ for $m \in M_0$.
\end{itemize}

We say that the scalar extension $(M, M_0) \boxtimes_K E$ is well-defined if $E \otimes_K M_0 \to M \boxtimes_K E$ is injective, in this case $(M, M_0) \boxtimes_K E$ is a $2$-step nilpotent $E$-module with the operations given by $(m \boxtimes e) \cdot e' = m \boxtimes ee'$ and $\tau(m \boxtimes e) = e^2 \otimes \tau(m)$, also $(M \boxtimes_K E) / (E \otimes_K M_0) \cong (M / M_0) \otimes_K E$ as a right $E$-module. The operations are well-defined by \cite[proposition 1.20]{Hartl}. If the scalar extension is well-defined, then is satisfies the following universal property: $\Hom_E((M, M_0) \boxtimes_K E, (N, N_0)) \cong \Hom_K((M, M_0), (N, N_0))$ for all $2$-step nipotent $E$-modules $(N, N_0)$, where $\Hom_K$ and $\Hom_E$ denote the sets of morphisms of $2$-step nilpotent modules. Let us call $(M, M_0)$ universal if all its scalar extensions are well-defined.

\begin{lemma}\label{2-mod-ext}
A $2$-step nilpotent $K$-module $(M, M_0)$ is universal in the following cases: it is a direct limit of universal $2$-step nilpotent $K$-modules, it split, the $K$-module $M / M_0$ is flat, or $(M, M_0) = (N / X, N_0 / X)$ for a universal $2$-step $K$-module $(N, N_0)$ and a left $K$-submodule $X \leq N_0$.
\end{lemma}
\begin{proof}
If $M = M_1 \dotoplus M_0$ is split, then $M \boxtimes_K E = (E \otimes_K M_1) \dotoplus (E \otimes_K M_0)$ with the $2$-cocycle $b_E$, hence $(M, M_0)$ is universal and its scalar extensions are also split. Clearly, a direct limit of universal $2$-step nilpotent $K$-modules is universal. If $M / M_0$ is flat, then by Govorov -- Lazard theorem $M / M_0 = \varinjlim_i P_i$ for some finite free $K$-modules $P_i$. It is easy to see that $(M \times_{M / M_0} P_i, M_0)$ are split $2$-step nilpotent $K$-modules with $(m, p) \cdot k = (m \cdot k, pk)$ and $\tau(m, p) = \tau(m)$ for $(m, p) \in M \times_{M / M_0} P_i$. Hence $(M, M_0) = \varinjlim_i (M \times_{M / M_0} P_i, M_0)$ is universal. The last case is trivial.
\end{proof}

There is a non-universal $2$-step nilpotent $K$-module. Let $K = \mathbb Z[\sqrt 2]$, $\widetilde M = K \dotoplus K$ be the commutative split $2$-step nilpotent $K$-module with $b(x, y) = xy$, and $M = \widetilde M / A$, where $A$ is the $K^\bullet$-invariant subgroup generated by $\sqrt 2 \dotoplus 1$. Clearly, $A$ is contained in the kernel of $\tau$ and has trivial intersection with $M_0$, hence $(M, M_0)$ is a $2$-step nilpotent $K$-module. It is easy to see that the map $\mathbb F_2 \otimes_K M_0 \to M \boxtimes_K \mathbb F_2$ is zero, since $(1 \dotoplus \sqrt 2) \cdot \sqrt 2 = 0 \dotoplus 1$ in $M$.

For any multiplicative subset $S \subseteq K$ and a $2$-step nilpotent $K$-module $(M, M_0)$ there is a functorial construction of the localization $(S^{-1} M, S^{-1} M_0)$ as a $2$-step nilpotent $S^{-1} K$-module, see \cite{UnitK2}. It is easy to see that it is the scalar extension of $(M, M_0)$ by $S^{-1} K$, so such scalar extensions are always well-defined. By the following lemma, this holds for all flat extensions.

\begin{lemma}
Let $(M, M_0)$ be a $2$-step $K$-module and $K \to E$ be a flat ring homomorphism. Then the scalar extension $(M, M_0) \boxtimes_K E$ is well-defined. If $K \to E$ is faithfully flat, then the morphism $(M, M_0) \to (M, M_0) \boxtimes_K E$ of $2$-step nilpotent $K$-modules is injective.
\end{lemma}
\begin{proof}
By Govorov -- Lazard theorem $E = \varinjlim_i P_i$ for finite free $K$-modules $P_i$. Let $G_i$ be the groups generated by $E \otimes_K M_0$ and $m \boxtimes p$ for $m \in M$, $p \in P_i$ with the same relations as for the extension of scalars, using the product maps $P_i \otimes_K P_i \to E$ instead of the multiplication on $E$. It is easy to see that $G_i = E \otimes_K M_0 \dotoplus \bigoplus_j^\cdot A \boxtimes e_j$, where $e_1, \ldots, e_{N_i}$ is the basis of $P_i$ and $A$ is any set of representatives of $M / M_0$ in $M$ containing $\dot 0$. Hence $E \otimes_K M_0$ is a subgroup of $G_i$ and their direct limit $\varinjlim_i G_i = M \boxtimes_K E$. The last claim follows from the injectivity of $M_0 \to E \otimes_K M_0$ and $M / M_0 \to (M / M_0) \otimes_K E$.
\end{proof}

Now let us generalize faithfully flat descent to $2$-step nilpotent modules. Let $K \to E$ be a faithfully flat homomorphism of commutative rings. A descent datum for $2$-step nilpotent modules with respect to $E / K$ is a $2$-step nilpotent $E$-module $(N, N_0)$ with an isomorphism of $2$-step nilpotent $(E \otimes_K E)$-modules $\psi \colon i_1^* (N, N_0) \cong i_2^* (N, N_0)$ satisfying the cocycle condition. Here $i_s^*(N, N_0)$ are the extensions of scalars of $(N, N_0)$ along one of the two canonical maps $E \to E \otimes_K E$.

\begin{prop}\label{2-mod-descent}
Let $E / K$ be a faithfully flat extension of commutative rings. Then the category of $2$-step nilpotent $K$-modules is equivalent to the category of descent data for $2$-step nilpotent modules with respect to $E / K$. A $2$-step nilpotent $K$-module $(M, M_0)$ corresponds to $(M, M_0) \boxtimes_K E$ with the canonical isomorphism $\psi$, it is universal if and only if $(M, M_0) \boxtimes_K E$ is universal over $E$.
\end{prop}
\begin{proof}
The construction from the statement is a functor from the category of $2$-step nilpotent $K$-modules to the category of descent data. We construct its right adjoint. Let $i_1, i_2 \colon E \to E \otimes_K E$; $i_{12}, i_{13}, i_{23} \colon E \otimes_K E \to E \otimes_K E \otimes_K E$; and $j_1, j_2, j_3 \colon E \to E \otimes_K E \otimes_K E$ be the canonical homomorphisms. Let $(N, N_0)$ be a $2$-step $E$-module and $\psi \colon i_1^*(N, N_0) \cong i_2^*(N, N_0)$ be an isomorphism satisfying the cocycle condition $i_{23}^* \psi \circ i_{12}^* \psi = i_{13}^* \psi \colon j_1^* N \to j_3^* N$. By $i_s$ we also denote the canonical morphisms of $2$-step modules $(N, N_0) \to i_s^*(N, N_0)$, and similarly for $i_{st}$ and $j_s$. Let $M = \{n \in N \mid \psi(i_1(n)) = i_2(n)\}$ and $M_0 = M \cap N_0$, then $(M, M_0)$ is a $2$-step $K$-module. We are going to prove that the map
$$f \colon M / M_0 \to \{n \in N / N_0 \mid \psi(i_1(n)) = i_2(n) \in i_2^*(N / N_0)\}$$
is an isomorphism. Clearly, $f$ is injective.

Let $n \in N$ be such that $x = \psi(i_1(n)) \dotminus i_2(n) \in i_2^*(N_0)$. In order to prove that $f$ is surjective it suffices to show that there is $\widetilde n \in N_0$ such that $x = \psi(i_1(\widetilde n)) - i_2(\widetilde n)$. This is equivalent to $(i_{23}^* \psi)(i_{12}(x)) + i_{23}(x) = i_{13}(x)$ by ordinary descent and by exactness of the Amitsur complex of $E / K$ tensored by $N_0$. After substitution we obtain
$$(i_{23}^* \psi \circ i_{12}^* \psi)(j_1(n)) \dotminus (i_{23}^* \psi)(j_2(n)) \dotplus (i_{23}^* \psi)(j_2(n)) \dotminus j_3(n) = (i_{13}^* \psi)(j_1(n)) \dotminus j_3(n),$$
i.e. the cocycle condition.

Clearly, the descent construction is right adjoint to the functor from the statement. It remains to show that the unit and the counit of our adjunction are isomorphisms. If $(M, M_0)$ is a $2$-step nilpotent $K$-module, then the unit of the adjunction is bijective on $M_0$ and $M / M_0$ by ordinary descent. Hence the unit is an isomorphism. The proof that the counit is an isomorphism is similar.

Finally, suppose that $(N, N_0)$ is a universal $2$-step nilpotent $E$-module and $\psi \colon i_1^*(N, N_0) \cong i_2^*(N, N_0)$ satisfies the cocycle condition. Let $K \to K'$ be a homomorphism of commutative rings and $E' = E \otimes_K K'$. If we denote the descent construction by $\mathrm{Des}^E_K(N, N_0; \psi)$, then the obvious morphism $\mathrm{Des}^E_K(N, N_0; \psi) \boxtimes_K K' \to \mathrm{Des}^{E'}_{K'}((N, N_0) \boxtimes_E E'; \psi_{E'})$ is an isomorphism, since it is bijective on the subgroups and the factor-groups. It follows that $\mathrm{Des}^E_K(N, N_0; \psi)$ is universal.
\end{proof}

If $(M, M_0)$ is a $2$-step nilpotent module over a product of rings $K' \times K''$, then there is a natural decomposition $(M, M_0) \cong (M', M_0') \times (M'', M_0'')$, where $(M', M_0')$ is a $2$-step nilpotent $K'$-module with the zero action of $(K'')^\bullet$ and similarly for $(M'', M_0'')$. Conversely, if $(M', M_0')$ is a $2$-step nilpotent $K'$-module and $(M'', M_0'')$ is a $2$-step nilpotent $K''$-module, then $(M', M_0') \times (M'', M_0'')$ is a $2$-step module over $K' \times K''$. The same decomposition holds for morphisms between $2$-step nilpotent modules.

It follows that if $(M, M_0)$ is a universal $2$-step nilpotent $K$-module, then the functor $E / K \mapsto M \boxtimes_K E$ from the category of commutative $K$-algebras to the category of groups satisfies the fpqc sheaf condition. In particular, this functor is a sheaf with respect to the fppf topology on the category of affine $K$-schemes. Now it is possible to define quasi-coherent $2$-step nilpotent sheaves on schemes and algebraic spaces, though we do not need this.

We are ready to apply the descent theory to quadratic modules and odd form algebras. Let $(R, \Delta, \mathcal D)$ be an augmented odd form $K$-algebra such that $(\Delta, \mathcal D)$ is universal, $E / K$ be a commutative ring extension. By \cite[proposition 1.20]{Hartl} and lemma \ref{2-mod-ext} the triple
$$E \otimes_K (R, \Delta, \mathcal D) = \bigl(R \otimes_K E, (\Delta \boxtimes_K E) / X, (E \otimes_K \mathcal D) / X\bigr)$$
is an augmented odd form $E$-algebra such that $\bigl((\Delta \boxtimes_K E) / X, (E \otimes_K \mathcal D) / X\bigr)$ is universal, where $X = \{\phi(a) \mid a = \inv a \in R \otimes_K E\}$. The operations for $a \in R$, $u \in \Delta$, $v \in \mathcal D$, $e, e' \in E$ are given by
\begin{itemize}
\item $(e \otimes v) \cdot (a \otimes e') = e {e'}^2 \otimes (v \cdot a)$, $(u \boxtimes e) \cdot (a \otimes e') = (u \cdot a) \boxtimes ee'$;
\item $\phi(a \otimes e) = e \otimes \phi(a)$;
\item $\pi(u \boxtimes e) = \pi(u) \otimes e$;
\item $\rho(e \otimes v) = \rho(v) \otimes e$, $\rho(u \boxtimes e) = \rho(u) \otimes e^2$.
\end{itemize}
Notice that $X = \dot 0$ if $E$ is flat over $K$ or if $\rho \colon \mathcal D \to R$ is a universally injective homomorphism of $K$-modules (i.e. it remains injective after any extension of scalars).

It follows that every augmented odd form $K$-algebra $(R, \Delta, \mathcal D)$ with universal $(\Delta, \mathcal D)$ gives a functor $(-) \otimes_K (R, \Delta, \mathcal D)$ on the category of commutative $K$-algebras. By proposition \ref{2-mod-descent}, this functor satisfies the fpqc sheaf condition.

It is not true in general that $E \otimes_K (R, \Delta, \mathcal D)$ is special for a special augmented odd form $K$-algebra $(R, \Delta, \mathcal D)$ even if $E / K$ is faithfully flat. Indeed, let $R = K$ be a perfect field of characteristic $2$, $\Delta = K$ with operations $x \cdot y = xy^2$ for $y \in R$, $\phi(x) = \dot 0$, $\rho(y) = y$, and $\pi(y) = 0$. This odd form $K$-algebra is special unital. Let $\mathcal D = 0$, so $(R, \Delta, \mathcal D)$ is an augmented odd form $K$-algebra and $(\Delta, \mathcal D)$ is universal. Now for a faithfully flat extension $E / K$ with $E = K[\eps] / (\eps^2)$ we have $\Delta \boxtimes_K E = E \dotoplus E \boxtimes \eps$, where $\rho(y \dotoplus y' \boxtimes \eps) = y$ and $\pi(y \dotoplus y' \boxtimes \eps) = 0$, i.e. $E \otimes_K (R, \Delta, \mathcal D)$ is not special.

On the other hand, let $(R, \Delta, \mathcal D)$ be a special augmented odd form $K$-algebra with universal $(\Delta, \mathcal D)$ and the additional property $\mathcal D = \Ker(\pi)$. Then $E \otimes_K (R, \Delta, \mathcal D)$ is special and has same property for all flat extensions $E / K$. If, in addition, the $K$-module homomorphisms $\pi \colon \Delta / \mathcal D \to R$ and $\rho \colon \mathcal D \to R$ are universally injective, then this holds for all extensions $E / K$.

All claims from the following lemma hold in the classical cases.

\begin{lemma}\label{canon-ext}
Let $(M, B, q)$ be a quadratic module over an even quadratic $K$-algebra $(R, L, A)$ with the odd form parameter $\mathcal L$. Suppose that the short exact sequence \(0 \to \Lambda \to L \to A \to 0\) is universal. Then $(\mathcal L, \Lambda)$ is a universal $2$-step nilpotent $K$-module, its construction commutes with scalar extensions.

Moreover, if $(R, L, A)$ is regular and $M$ is flat over $R$, then the augmented odd form $K$-algebra $(S, \Theta, \mathcal F)$ given by the canonical construction has the following properties: $\pi \colon \Theta / \mathcal F \to S$ is an isomorphism and $\rho \colon \mathcal F \to S$ is a universally injective homomorphism. The canonical construction commutes with scalar extensions.
\end{lemma}
\begin{proof}
The pair $(\mathcal L, \Lambda)$ is a $2$-step nilpotent $K$-module with $\tau(m, l) = l - \inv{\,l\,} \in \Lambda$. There is a natural morphism \((\mathcal L, \Lambda) \boxtimes_K E \to (\mathcal L_E, \Lambda \otimes_K E)\) for all commutative $K$-algebras $E$, where $\mathcal L_E$ is the odd form parameter of $(E \otimes_K M, B_E, q_E)$. Since it is bijective on the subgroups and the factor-groups of the filtrations, it is an isomorphism and the scalar extension is well-defined.

To prove the second claim, by Govorov -- Lazard theorem applied to $N = L^{-1} \otimes_{R^\oppose} M^\oppose$ we may assume that $N$ is free. Then the result follows from lemma \ref{2-mod-ext} and the explicit description of $(\Theta, \mathcal F)$ from lemmas \ref{canon-well-def} and \ref{canon-aug}.
\end{proof}

\section{Classical odd form algebras}

Here we describe the odd form algebras and their unitary groups obtained from the canonical construction in the classical cases.

Let $M$ be the split quadratic module of rank $n$ over the classical even quadratic $K$-algebra of linear type (i.e. of rank $2n$ over $K$). Its canonical augmented odd form algebra $(R, \Delta, \mathcal D) = \ofalin(n, K)$ is called a linear odd form $K$-algebra. Using the naive construction and proposition \ref{canon-aug}, we have $R = \bigoplus_{\substack{1 \leq |i|, |j| \leq n\\ ij > 0}} K e_{ij}$, $\mathcal D = \bigoplus_{1 \leq i, j \leq n} K \phi(e_{ij})$, and $\Delta = \bigoplus^\cdot_{\substack{1 \leq |i|, |j| \leq n\\ ij > 0}} q_i \cdot K e_{ij} \dotoplus \mathcal D$ with
\begin{itemize}
\item $\inv{e_{ij}} = e_{-j, -i}$, $e_{ij} e_{jl} = e_{il}$, $e_{ij} e_{kl} = 0$ for $j \neq k$;
\item $\pi(q_i) = e_{ii}$, $\rho(q_i) = 0$, $q_i = q_i \cdot e_{ii}$.
\end{itemize}
By the naive construction, $\unit(\ofalin(n, K)) \cong \unit(M, B_M, q_M) = \glin(n, K)$ is the split general linear group. There is a monoid homomorphism $\det \colon R^\bullet \to (K \times K)^\bullet, (r_1, r_2) \mapsto (\det(r_1), \det(r_2))$ such that
$$\slin(n, K) \cong \{g \in \unit(\ofalin(n, K)) \mid \det(\alpha(g)) = 1\}.$$
Note that $K \cong \{x \in \Cent(R) \mid x = \inv x\}$ for $n > 0$.

Now let $M$ be the split symplectic $K$-module of rank $2n$. Its canonical augmented odd form algebra $(R, \Delta, \mathcal D) = \ofasymp(2n, K)$ is called a symplectic odd form $K$-algebra. Again using the naive construction and proposition \ref{canon-aug}, we obtain $R = \bigoplus_{1 \leq |i|, |j| \leq n} K e_{ij}$, $\mathcal D = \bigoplus_{\substack{1 \leq |i|, |j| \leq n\\ i + j > 0}} K \phi(e_{ij}) \oplus \bigoplus_{1 \leq |i| \leq n} Kv_i$, and $\Delta = \bigoplus^\cdot_{1 \leq |i|, |j| \leq n} q_i \cdot Ke_{ij} \dotoplus \mathcal D$ with
\begin{itemize}
\item $\inv{e_{ij}} = \eps_i \eps_j e_{-j, -i}$, $e_{ij} e_{jl} = e_{il}$, $e_{ij} e_{kl} = 0$ for $j \neq k$;
\item $\phi(e_{-ii}) = 2v_i$, $\pi(v_i) = 0$, $\rho(v_i) = e_{-i, i}$;
\item $v_i \cdot e_{ik} = \eps_i \eps_k v_k$, $v_i \cdot e_{jk} = \dot 0 \text{ for } i \neq j$;
\item $\pi(q_i) = e_{ii}$, $\rho(q_i) = 0$, $q_i = q_i \cdot e_{ii}$.
\end{itemize}
Here $\unit(\ofasymp(2n, K)) \cong \unit(M, B_M, q_M) = \symp(2n, K)$ is the split symplectic group. There is a canonical isomorphism $K \cong \Cent(R)$ for $n > 0$.

For the linear and symplectic odd form algebras the odd form parameter $\Delta$ is actually the maximal possible (for a special odd form ring). Hence their unitary groups are determined by the algebra $R$, it is an Azumaya algebra with involution. In the linear case the involution is of the second kind (over the ring $K \times K$), and in the symplectic case the involution is symplectic. The Azumaya algebras with orthogonal involutions arise from the orthogonal odd form algebras, as we show below (in the odd rank case we need $2 \in K^*$). But in the orthogonal odd form algebras the odd form parameter is not maximal if $2 \notin K^*$, hence the unitary group is not determined by the algebra $R$.

Consider the split classical quadratic $K$-module $M$ of rank $2n$. Its canonical augmented odd form algebra $(R, \Delta, \mathcal D) = \ofaorth(2n, K)$ is called an orthogonal odd form $K$-algebra (of even rank). Again using the naive construction and proposition \ref{canon-aug}, we have $R = \bigoplus_{1 \leq |i|, |j| \leq n} K e_{ij}$, $\mathcal D = \bigoplus_{\substack{1 \leq |i|, |j| \leq n\\ i + j > 0}} K \phi(e_{ij})$, and $\Delta = \bigoplus^\cdot_{1 \leq |i|, |j| \leq n} q_i \cdot Ke_{ij} \dotoplus \mathcal D$ with
\begin{itemize}
\item $\inv{e_{ij}} = e_{-j, -i}$, $e_{ij} e_{jl} = e_{il}$, $e_{ij} e_{kl} = 0$ for $j \neq k$;
\item $\pi(q_i) = e_{ii}$, $\rho(q_i) = 0$, $q_i = q_i \cdot e_{ii}$.
\end{itemize}
Clearly, $\unit(\ofaorth(2n, K)) \cong \unit(M, B_M, q_M) = \orth(2n, K)$ is the split orthogonal group. Let $(\mathbb Z / 2 \mathbb Z)(K)$ be the group of idempotents of $K$ with the group operation $e * f = e + f - 2 ef$. There is a group homomorphism $\Dickson \colon \orth(2n, K) \to (\mathbb Z / 2 \mathbb Z)(K)$ such that
$$\sorth(2n, K) = \{g \in \orth(2n, K) \mid \Dickson(g) = 1\}$$
(the Dickson invariant, it satisfies $\det(g) = 1 - 2 \Dickson(g)$), see \cite[section IV.5]{Knus} or \cite[Appendix C]{Conrad} for details. As in the symplectic case, there is a canonical isomorphism $K \cong \Cent(R)$ for $n > 0$.

Finally, let $M$ be the split classical orthogonal $K$-module of rank $2n + 1$. Its canonical augmented odd form algebra $(R, \Delta, \mathcal D) = \ofaorth(2n + 1, K)$ is called an orthogonal odd form $K$-algebra of odd rank. Since $M$ is not regular if $2 \notin K^*$, we cannot apply proposition \ref{canon-aug} and the unitary group $\unit(\ofaorth(2n + 1, K))$ may differs from $\unit(M, B_M, q_M) = \orth(2n + 1, K)$, as we show below. Let $e_{ij} = e_i e_{-j}^\oppose \in R$ for all $i, j$; $u_i = (e_0, -1) \boxtimes e_{-i}^\oppose \in \mathcal D$ for all $i$; and $q_i = (e_i, 0) \boxtimes e_{-i}^\oppose \in \Delta$ for $i \neq 0$. We get $R = \bigoplus_{-n \leq i, j \leq n} Ke_{ij}$, $\mathcal D = \bigoplus_{\substack{-n \leq i, j \leq n\\ i + j > 0}} K \phi(e_{ij})$, and $\Delta = \bigoplus^\cdot_{\substack{-n \leq i, j \leq n\\ i \neq 0}} q_i \cdot Ke_{ij} \dotoplus \bigoplus^\cdot_{-n \leq i \leq n} u_i \cdot K \dotoplus \mathcal D$ with
\begin{itemize}
\item $\inv{e_{ij}} = e_{-j, -i}$, $e_{ij} e_{jl} = e_{il}$ for $j \neq 0$, $e_{i0} e_{0l} = 2e_{il}$, $e_{ij} e_{kl} = 0$ for $j \neq k$;
\item $\pi(u_i) = e_{0i}$, $\rho(u_i) = -e_{-i, i}$;
\item $u_i \cdot e_{ik} = u_k$ for $i \neq 0$, $u_0 \cdot e_{0k} = u_k \cdot 2$, $u_i \cdot e_{jk} = \dot 0$ for $i \neq j$;
\item $\pi(q_i) = e_{ii}$, $\rho(q_i) = 0$, $q_i = q_i \cdot e_{ii}$.
\end{itemize}
The morphism into the naive odd form algebra is given by
\begin{align*}
\rep \colon R &\to \mat(2n + 1, K)^\oppose \times \mat(2n + 1, K),\\
e_{ij} &\mapsto (e_{-j, -i}^\oppose, e_{ij}) \text{ for } j \neq 0,\\
e_{i0} &\mapsto (2 e_{0, -i}^\oppose, 2 e_{i0}).
\end{align*}
The kernel of this ring homomorphism is $\{ke_{00} \mid 2k = 0\}$.

Let $(R, \Delta, \mathcal D) = \ofaorth(2n + 1, K)$ and $\Delta_{C(R)} = \{u \in \Delta \mid \pi(u) \in \Cent(R), \rho(u) \in \Cent(R)\}$. It is easy to see that $\Cent(R) = \{x(k) \mid k \in K\}$ and $\Delta_{\Cent(R)} = \{u(k) \mid k \in K\}$, where $x(k) = k e_{00} + 2k \sum_{i \neq 0} e_{ii}$ and $u(k) = \sum^\cdot_{i \neq 0} q_i \cdot 2k \dotplus u_0 \cdot k \dotplus \phi(2k^2 \sum_{i > 0} e_{ii})$. Hence both $\Cent(R)$ and $\Delta_{\Cent(R)}$ are isomorphic to $K$ as abelian groups, $x(k)\, x(k') = x(2kk')$, $\pi(u(k)) = x(k)$, $\rho(u(k)) = x(-k^2)$, $\phi(x(k)) = \dot 0$, and $u(k) \cdot x(k') = u(2kk')$. In other words, we may recover $K$ as an abelian group with the operation $k \mapsto k^2$ from the odd form ring $(R, \Delta)$.

The following proposition gives a complete description of the group $\unit(\ofaorth(2n + 1, K))$. Note that it is not isomorphic to $\orth(2n + 1, K) \cong \sorth(2n + 1, K) \times \mu_2(K)$ in general, where $\sorth(2n + 1, K) = \{g \in \orth(2n + 1, K) \mid \det(g) = 1\}$ and $\mu_2(K)$ is the group of square roots of $1$ in $K$. We denote the group $\unit(\ofaorth(2n + 1, K))$ by $\widetilde{\orth}(2n + 1, K)$, it is functorial on $K$. The idea of the embedding $\ofaorth(2n + 1, K) \to \ofaorth(2n + 2, K)$ is taken from \cite{DefMamaev}.

\begin{prop}\label{so-odd}
Let $M$ be the split classical orthogonal $K$-module of rank $2n + 1$ and $(R, \Delta, \mathcal D) = \ofaorth(2n + 1, K)$ be its canonical augmented odd form algebra. Then the functor $\widetilde{\orth}(2n + 1, -)$ on the category of commutative $K$-algebras is canonically isomorphic to $\sorth(2n + 1, -) \times \mathbb Z / 2 \mathbb Z$, it a smooth group scheme of relative dimension $n(2n + 1)$. In this decomposition over $K$ the second projection $\Dickson \colon \widetilde{\orth}(2n + 1, K) \to (\mathbb Z / 2 \mathbb Z)(K)$ satisfies $\det(1 + \rep(\beta(g))) = 1 - 2 \Dickson(g)$ and the kernel of the first projection is $\{g \in \widetilde{\orth}(2n + 1, K) \mid \beta(g) \in \Cent(R)\}$.
\end{prop}
\begin{proof}
There is an embedding $M \to M'$ into the split classical orthogonal $K$-module of rank $2n + 2$ given by $e_i \mapsto e_i$ for $i \neq 0$ and $e_0 \mapsto e_{-n - 1} + e_{n + 1}$. This embedding induces a morphism between special augmented odd form algebras
\begin{align*}
f \colon \ofaorth(2n + 1, K) &\to \ofaorth(2n + 2, K),\\
e_{ij} &\mapsto e_{ij} \text{ for } i, j \neq 0,\\
e_{i0} &\mapsto e_{i, -n - 1} + e_{i, n + 1} \text{ for } i \neq 0,\\
e_{0j} &\mapsto e_{-n - 1, j} + e_{n + 1, j} \text{ for } j \neq 0,\\
e_{00} &\mapsto e_{-n - 1, -n - 1} + e_{-n - 1, n + 1} + e_{n + 1, -n - 1} + e_{n + 1, n + 1}.
\end{align*}

In particular, $f$ is injective. Let $\ofaorth(2n + 2, K) = (T, \Xi, \mathcal X)$. It is easy to see that $f(R) = \{t \in T \mid t(e_{-n - 1} - e_{n + 1}) = \inv{\,t\,} (e_{-n - 1} - e_{n + 1}) = 0\}$, $f(\mathcal D) = \{v \in \mathcal X \mid \rho(v) \in f(R)\}$, and $f(\Delta) = \{u \in \Xi \mid \pi(u), \rho(u) \in f(R)\}$. Here $T$ acts on $K^{2n + 2}$ as in the naive construction. Hence
$$\widetilde \orth(2n + 1, K) \cong \{g \in \orth(2n + 2, K) \mid g(e_{-n - 1} - e_{n + 1}) = e_{-n - 1} - e_{n + 1}\}.$$
The map $\Dickson \colon \widetilde \orth(2n + 1, K) \to (\mathbb Z / 2 \mathbb Z)(K)$ is the restriction of the Dickson invariant for $\orth(2n + 2, K)$, it clearly satisfies $\det\bigl(1 + \rep(\beta(g))\bigr) = \det\bigl(\alpha(f(g))\bigr) = 1 - 2 \Dickson(g)$. Note that $\rep$ is the action on the orthogonal complement of $e_{n + 1} - e_{-n - 1}$.

The functor $\widetilde \orth(2n + 1, -)$ is given by the equations
\begin{align*}
\sum_{k = -n}^n \beta_{-k, -i} \beta_{kj} + \beta_{0, -i} \beta_{0j} + \beta_{ij} + \beta_{-j, -i} &= 0 \text{ for } i + j > 0,\\
\sum_{k = 0}^n \beta_{-k, i} \beta_{ki} + \beta_{-i, i} &= 0 \text{ for all } i
\end{align*}
on the variable $\beta_{ij}$, where $\beta(g) = \sum_{i, j} \beta_{ij} e_{ij}$, so it is representable. By the Jacobian criterion, $\widetilde \orth(2n + 1, -)$ is smooth over $K$ near the identity section. It follows that $\widetilde \orth(2n + 1, -)$ is smooth of relative dimension $n(2n + 1)$ over any field. Hence the differentials of the equations are linearly independent at every point, i.e. $\widetilde \orth(2n + 1, -)$ is representable by a smooth affine group scheme.

Let $\mathrm Z(K) = \{g \in \widetilde \orth(2n + 1, K) \mid \beta(g) \in \Cent(R)\} = \{(x(k), u(k)) \mid k^2 + k = 0\}$. Since $\Dickson(g) = -k$ for $g = (x(k), u(k)) \in \mathrm Z(K)$, we have a decomposition $\widetilde \orth(2n + 1, K) \cong \Ker(\Dickson) \times (\mathbb Z / 2 \mathbb Z)(K)$. The map $\rep$ induces a morphism between smooth group schemes
$$p \colon \Ker\bigl(\Dickson \colon \widetilde \orth(2n + 1, -) \to \mathbb Z / 2 \mathbb Z\bigr) \to \sorth(2n + 1, -).$$
By the fibral isomorphism criterion, it suffices to show that $p$ is an isomorphism if $K$ is a field. It is easy to see that the differential in the neutral element $d_e p$ is a bijection between the Lie algebras, hence the morphism $p$ is an isogeny with \'etale kernel (recall that $\sorth(2n + 1, -)$ is connected if $K$ is a field). But over arbitrary $K$ any element $g$ from the kernel of $p(K)$ is of type $\beta(g) = k e_{00}$ with $2k = 0$ and $k^2 + k = 0$. Such an element necessary lies in $\mathrm Z(K)$, i.e. is trivial.
\end{proof}

The twisted forms of $\ofalin(n, K^m)$, $\ofasymp(2n, K^m)$, and $\ofaorth(n, K^m)$ in the fppf-topology over $K$ are called classical odd form $K$-algebras. They are augmented and special. A classical odd form $K$-algebra is called split if it is isomorphic to $\ofalin(n, K^m)$, $\ofasymp(2n, K^m)$, or $\ofaorth(n, K^m)$ for some $n, m \geq 0$.

\section{Automorphisms}

Recall that the classical projective group schemes are the factor-groups $\pglin(n, -) = \glin(n, -) / \glin(1, -) = \slin(n, -) / \mu_n$, $\psorth(2n, -) = \sorth(2n, -) / \mu_2$, $\psymp(2n, -) = \symp(2n, -) / \mu_2$ in the sense of group schemes, i.e. as fppf sheaves on the category of $K$-schemes. Also $\psorth(2n + 1, K) = \sorth(2n + 1, K)$.

Let $(R, \Delta)$ be an odd form $K$-algebra. Its projective unitary group $\punit(R, \Delta)$ is the group of its automorphisms over $K$. Similarly, a projective unitary group $\punit(R, \Delta, \mathcal D)$ of an augmented odd form $K$-algebra is its automorphism group, this group satisfies the fpqc sheaf condition if $(\Delta, \mathcal D)$ is universal. If $(R, \Delta, \mathcal D)$ is a special augmented odd form $K$-algebra and $\mathcal D = \Ker(\pi)$ (as in the classical cases), then $\punit(R, \Delta) = \punit(R, \Delta, \mathcal D)$.

There is a natural homomorphism $\unit(R, \Delta) \to \punit(R, \Delta)$, an element $g \in \unit(R, \Delta)$ acts on $(R, \Delta)$ by $\up ga = \alpha(g)\, a\, \alpha(g)^{-1}$ for $a \in R$ and $\up gu = (\gamma(g) \cdot \pi(u) \dotplus u) \cdot \alpha(g)^{-1}$ for $u \in \Delta$. Note that if $(R, \Delta, \mathcal D)$ is an augmented odd form algebra, then the action of $\unit(R, \Delta)$ on $(R, \Delta)$ preserves $\mathcal D$ as a $K$-module, hence we actually have a homomorphism $\unit(R, \Delta) \to \punit(R, \Delta, \mathcal D)$. Also this action gives the ordinary conjugation action of $\unit(R, \Delta)$ on itself.

In order to determine the projective unitary groups of classical odd form algebras we use the classification of automorphisms of reductive group schemes. Let $(X^*, \Phi, X_*, \Phi^\vee)$ be a root datum, $\Delta \subseteq \Phi$ be a base, $G = G(X^*, \Phi, X_*, \Phi^\vee)$ be the corresponding split reductive group scheme over $K$, and $\mathbf{Aut}(G)$ be the automorphism group sheaf of $G$. Then $\mathbf{Aut}(G)$ is representable, the canonical homomorphism $G \to \mathbf{Aut}(G)$ has kernel $\mathbf C(G)$ (the scheme-theoretic center of $G$), and $\mathbf{Aut}(G) \cong G / \mathbf C(G) \rtimes \Aut(X^*, \Phi, X_*, \Phi^\vee; \Delta)$, where the second factor is considered as a discrete group scheme. It is proved in \cite[theorem 7.1.9]{Conrad} and \cite[Exp. XXIV, theorem 1.3]{SGA3}. We denote the abstract group $\Aut(X^*, \Phi, X_*, \Phi^\vee; \Delta)$ by $\mathrm{Out}(G)$ and call it the outer automorphism group of $G$.

The group schemes $\glin(n, -)^m$, $\slin(n, -)^m$, $\sorth(n, -)^m$, $\symp(2n, -)^m$, $\pglin(n, -)^m$, $\psorth(2n, -)^m$, and $\psymp(2n, -)^m$ are reductive for all $n, m \geq 0$. Most of them are semi-simple, hence their outer automorphism groups are subgroups of the automorphism groups of their Dynkin diagrams (certain powers of $\mathsf A_\ell$, $\mathsf B_\ell$, $\mathsf C_\ell$, or $\mathsf D_\ell$). The exceptions are $\glin(n, -)^m$ for $n, m \geq 1$ and the small rank cases $\sorth(2, -)^m$, $\psorth(2, -)^m$ for $m \geq 1$. Obviously, $G(K)^m \cong G(K^m)$ for any group scheme $G$.

The root datum of $\glin(n, -)$ for $n \geq 1$ is given by $X^* = \bigoplus_{i = 1}^n \mathbb Z \mathrm f_i$ and $X_* = \bigoplus_{i = 1}^n \mathbb Z \mathrm e_i$ with $\langle \mathrm e_i, \mathrm f_i \rangle = 1$ and $\langle \mathrm e_i, \mathrm f_j \rangle = 0$ for $i \neq j$. The base $\Delta$ consists of $\alpha_i = \mathrm f_i - \mathrm f_{i + 1}$ for $1 \leq i \leq n - 1$, $\alpha_i^\vee = \mathrm e_i - \mathrm e_{i + 1}$.

\begin{lemma}\label{gl-out}
There is an outer automorphism $\sigma$ of $\glin(n, -)$ for $n \geq 1$ given by $\sigma(\mathrm e_i) = - \mathrm e_{n + 1 - i}$ and $\sigma(\mathrm f_i) = - \mathrm f_{n + 1 - i}$, hence $\sigma(\alpha_i) = \alpha_{n - i}$. It is induced by an automorphism of $\ofalin(n, K)$ of order $2$. Moreover, $\mathrm{Out}(\glin(n, -)) \cong \mathbb Z / 2 \mathbb Z$ is generated by $\sigma$.
\end{lemma}
\begin{proof}
The corresponding automorphism of the augmented odd form algebra $(R, \Delta, \mathcal D) = \ofalin(n, K)$ is given by $\sigma(e_{ij}) = e_{i \pm (n + 1), j \pm (n + 1)}$, it is the non-trivial involution on the center $\Cent(R) \cong K \times K$ of $R$. If $\tau$ is an outer automorphism of $\glin(n, -)$, then we may assume that $\tau(\alpha_i) = \alpha_i$ multiplying by $\sigma$ if necessary, since the Dynkin diagram $\mathsf A_{n - 1}$ is empty or a path graph. Then $\tau$ stabilizes $\Phi$ and $\Phi^\vee$, hence it is determined by the values on $\mathrm e_1$ and $\mathrm f_1$. Since $\tau$ preserve the pairing, we have $\tau(\mathrm e_1) = \mathrm e_1 + \lambda \sum_{i = 1}^n \mathrm e_i$ and $\tau(\mathrm f_1) = \mathrm f_1 + \mu \sum_{i = 1}^n \mathrm f_i$ for some integers $\lambda$ and $\mu$ such that $\lambda + \mu + n \lambda \mu = 0$. It follows that $\lambda$ and $\mu$ divide each other, i.e. $\lambda = \pm \mu$. If $\lambda = -\mu$, then necessarily $\lambda = \mu = 0$ and $\tau = \id$. If $\lambda = \mu \neq 0$, then $2 + n \lambda = 0$, hence $n = 1$ or $n = 2$. It is easy to see that in both cases $\tau = \sigma$.
\end{proof}

Recall a definition of hyperbolic pairs from \cite{UnitK2, OFADefVor}. We say that $\eta = (e_-, e_+, q_-, q_+)$ is a hyperbolic pair in an odd form ring $(R, \Delta)$ if $e_-$ and $e_+$ are orthogonal idempotents in $R$, $\inv{e_+} = e_-$, $q_-$ and $q_+$ lie in $\Delta$, $\pi(q_-) = e_-$, $\rho(q_-) = 0$, $q_- \cdot e_- = q_-$, $\pi(q_+) = e_+$, $\rho(q_+) = 0$, $q_+ \cdot e_+ = q_+$. An orthogonal hyperbolic family $\eta_1, \ldots, \eta_n$ of rank $n$ is a sequence of hyperbolic pairs $\eta_i = (e_{-i}, e_i, q_{-i}, q_i)$ such that $e_{|i|} = e_{-i} + e_i$ are pairwise orthogonal idempotents and $e_{|i|} \in R e_{|j|} R$ for all $1 \leq i, j \leq n$. The augmented odd form $K$-algebras $\ofalin(n, K)$, $\ofasymp(2n, K)$, $\ofaorth(2n + 1, K)$, and $\ofaorth(2n, K)$ have standard orthogonal hyperbolic families of rank $n$ with $\eta_i = (e_{-i, -i}, e_{ii}, q_{-i}, q_i)$. We use the notation $e_i = e_{ii}$ for the split classical odd form $K$-algebras.

Now recall the definitions of elementary transvections and dilations for an odd form ring $(R, \Delta)$ with an orthogonal hyperbolic family $\eta_1, \ldots, \eta_n$. Let $\Delta^0 = \bigl\{u \in \Delta \mid \bigl(\sum_{i = 1}^n |e_i|\bigr)\, \pi(u) = 0\bigr\} \leq \Delta$. An elementary transvection of a short root type is an element $T_{i j}(x) \in \unit(R, \Delta)$ such that
$$\beta(T_{i j}(x)) = x - \inv x, \quad\gamma(T_{ij}(x)) = q_i \cdot x \dotminus q_{-j} \cdot \inv x \dotminus \phi(x)$$
for any $i \neq 0$, $j \neq 0$, $i \neq \pm j$, and $x \in e_i R e_j$. An elementary transvection of an ultrashort root type is an element $T_i(u) \in \unit(R, \Delta)$ such that
$$\beta(T_i(u)) = \rho(u) + \pi(u) - \inv{\pi(u)}, \quad\gamma(T_i(u)) = u \dotminus \phi(\rho(u) + \pi(u)) \dotplus q_{-i} \cdot (\rho(u) - \inv{\pi(u)})$$
for any $i \neq 0$ and $u \in \Delta^0 \cdot e_i$. An elementary dilation is an element $D_i(a) \in \unit(R, \Delta)$ such that
$$\beta(D_i(a)) = a + \inv a^{-1} - e_{|i|}, \quad\gamma(D_i(a)) = q_{-i} \cdot (\inv a^{-1} - e_{-i}) \dotplus q_i \cdot (a - e_i) \dotminus \phi(a - e_i)$$
for any $i \neq 0$ and $a \in (e_i R e_i)^*$, here $e_i R e_i$ is a unital associative ring. Also, $D_0(g) = g$ for $g \in \unit(e_0 R e_0, \Delta^0 \cdot e_0)$. All these elements actually lie in the unitary group $\unit(R, \Delta)$.

Clearly, in the cases $(R, \Delta, \mathcal D) = \ofaorth(2n, K)$ or $(R, \Delta, \mathcal D) = \ofaorth(2n + 1, K)$ the elementary transvections and dilations except $D_0(-)$ lie in the special orthogonal group. In the case $(R, \Delta, \mathcal D) = \ofalin(n, K)$ all elementary transvections lie in the special linear group, a product $D_1(a_1 e_{11}) \cdots D_n(a_n e_{nn})$ lie in $\slin(n, K)$ if and only if $a_1 \cdots a_n = 1$.

\begin{lemma}\label{density-torus}
If $k_1, \ldots, k_n$ are distinct nonzero integers, then the elements $(a^{k_1} - 1, \ldots, a^{k_n} - 1)$ for $a \in E^*$ generate the fppf sheaf $E \mapsto E^n$ as a sheaf of non-unital algebras.
\end{lemma}
\begin{proof}
Let $\mathcal F$ be the subsheaf generated by these elements. If suffices to show that $\mathcal F(K)$ contains the standard idempotents of $K^n$. Without loss of generality, $n = 2$ and $k_1 < k_2$. Consider the fppf extension $E = K[\theta, \theta^{-1}] / (\theta^N - (\theta^{k_1} - 1) (\theta^{k_2} - \theta^{k_1}))$ for sufficiently large $N$. In this extension $\theta^{k_1} - 1$ and $\theta^{k_2} - \theta^{k_1}$ are invertible, hence $(1, 1 + u) \in \mathcal F(E)$ for $u = \frac{\theta^{k_2} - \theta^{k_1}}{\theta^{k_1} - 1} \in E^*$. It follows that $(a^{k_1}, a^{k_2} + u) \in \mathcal F(E')$ for any extension $E' / E$ and $a \in {E'}^*$. But for $E' = E[\psi, \psi^{-1}] / (\psi^{k_2} - u)$ we have $(1, 0) \in \mathcal F(E')$, hence $(1, 0), (0, 1) \in \mathcal F(K)$.
\end{proof}

\begin{lemma}\label{density}
Let $G$ be one of the group schemes $\glin(n, -)^m$, $\slin(n, -)^m$, $\sorth(n, -)^m$, or $\symp(2n, -)^m$ over $K$. Let $(R, \Delta, \mathcal D)$ be the corresponding split classical odd form $K$-algebra. Then $G$ generates the sheaf $(-) \otimes_K (R, \Delta)$ of odd form algebras unless $G$ is on of the exceptions $\sorth(1, -)^m$, $\sorth(2, -)^m$, $\slin(1, -)^m$, and $\slin(2, -)^m$ for $m \geq 1$.
\end{lemma}
\begin{proof}
Without loss of generality, $m = 1$ and $n > 0$. Let $\bigl(S(-), \Theta(-)\bigr)$ be the subsheaf of $(-) \otimes_K (R, \Delta)$ generated by $G$. Suppose at the moment that $G$ is not a special linear group scheme. Note that since $D_i(a e_i) \in G(E)$ for all $a \in E^*$ and $i \neq 0$, we have $(a^{-1} - 1) e_{-i} + (a - 1) e_i \in S(E)$ for any extension $E / K$. By lemma \ref{density-torus} both $e_{-i}$ and $e_i$ lie in $S(K)$ for $i \neq 0$.

If $G = \slin(n, -)$ and $n \geq 3$, then $\beta(T_{ij}(xe_{ij})) = xe_{ij} - x e_{-j, -i} \in S(K)$ for all $i j > 0$, $i \neq j$, $x \in K$. Choose an index $k$ different from $i$ and $j$ but with the same sign, then $(e_{ik} - e_{-k, -i}) (e_{kj} - e_{-j, -k}) = e_{ij} \in S(K)$ and $e_i = e_{ij} e_{ji} \in S(K)$. For simplicity, let $e_{ij} = 0$ in both linear cases for $ij < 0$.

Next consider the elementary transvections $T_{ij}(x e_{ij})$ for arbitrary $G$. It follows that $e_{ij} + \eps e_{-j, -i} \in S(K)$ for all $i \neq \pm j$ and $i, j \neq 0$, here $\eps = \pm 1$ depends on $G$. Hence $e_{ij} \in S$ for all $i, j \neq 0$ and $i \neq -j$. Also $e_{-i, i} = e_{-i, j} e_{ji} \in S(K)$ if the exists an index $j$ different from $0$ and $\pm i$. Such an index does not exist in the cases $\glin(1, -)$ (where we do not need $e_{-i, i}$), $\symp(2, -)$, $\sorth(2, -)$, and $\sorth(3, -)$. For $\symp(2n, -)$ using the elementary transvections $T_i(v_i \cdot x)$ we obtain $e_{-i, i} \in S(K)$ for all $i \neq 0$. For $\sorth(2n + 1, -)$ similarly using $T_i(u_i \cdot x)$ we get $e_{0i}, e_{i0}, e_{-i, i} \in S(K)$ for all $i \neq 0$ (and, consequently, $e_{00} \in S(K)$).

Since $T_{ij}(x e_{ij}) \in G(K)$, we also have $q_i \in \Theta(K)$ for all $i \neq 0$. It remains to note that $v_i \in \Theta(K)$ if $G = \symp(2n, -)$ and $u_i \in \Theta(K)$ if $G = \sorth(2n + 1, -)$ for all $i \neq 0$, because $T_i(u) \in \Theta(K)$ for all admissible $u$. In the case of $\sorth(2n + 1, -)$ we finally have $u_0 = u_1 \cdot e_{10} \in \Theta(K)$.
\end{proof}

The group schemes $\sorth(1, -)$ and $\slin(1, -)$ are trivial and generate zero sheaves of odd form subalgebras in $\ofaorth(1, -)$ and $\ofalin(1, -)$. There are obvious injective morphisms of sheaves of augmented odd form algebras $\ofasymp(2, -) \to \ofalin(2, -)$ and $\ofalin(1, -) \to \ofaorth(2, -)$ inducing the isomorphisms $\symp(2, -) \cong \slin(2, -)$ and $\glin(1, -) \cong \sorth(2, -)$, hence the group schemes $\slin(2, -)$ and $\sorth(2, -)$ generate the images of these morphisms.

The following lemma implies that every non-trivial classical odd form $K$-algebra is a ``simple'' classical odd form $E$-algebra (i.e. a twisted form of $\ofalin(n, E)$, $\ofasymp(2n, E)$, or $\ofaorth(n, E)$) over a uniquely determined finite \'etale extension $E$ of $K$, see the description of centers of split classical odd form algebras.

\begin{lemma}\label{fin-et-aut}
Let $K$ be a commutative ring and $m \geq 0$. Then the automorphism group of the $K$-module $K^m$ with the operation $(x_i)_i \mapsto (x_i^2)_i$ is isomorphic to the group of $K$-points of the discrete group scheme $\mathrm S_m$, it coincides with the automorphism group of the $K$-algebra $K^m$.
\end{lemma}
\begin{proof}
If $g \in \glin(n, K)$ preserves this operation, then $g_{ij}^2 = g_{ij}$ for all $i, j$ and $g_{ij} g_{ik} = 0$ for $j \neq k$. Hence $g$ is a monomial matrix locally in the Zariski topology.
\end{proof}

\begin{theorem}\label{ofa-aut}
Let $G$ be one of the group schemes $\glin(n, -)^m$, $\slin(n, -)^m$, $\symp(2n, -)^m$, $\sorth(n, -)^m$, $\pglin(n, -)^m$, $\psymp(2n, -)^m$, or $\psorth(n, -)^m$ over a commutative ring $K$. Let $(R, \Delta, \mathcal D)$ be the corresponding split classical odd form $K$-algebra. Then the canonical map $\punit(R, \Delta) \to \Aut(G)$ is an isomorphism unless $G$ is one of the exceptions $\sorth(2, -)^m$, $\psorth(2, -)^m$, $\psorth(8, -)^m$, $\slin(1, -)^m$, $\slin(2, -)^m$, $\pglin(1, -)^m$, $\pglin(2, -)^m$ for $m \geq 1$ or $\glin(n, -)^m$, $\sorth(1, -)^m$, $\psorth(1, -)^m$, $\psorth(4, -)^m$ for $n \geq 1$, $m \geq 2$.
\end{theorem}
\begin{proof}
In the cases $G = \sorth(1, -)$ and $G = \psorth(1, -)$ we cannot apply lemma \ref{density}, but then the group schemes $G$, $\mathbf{Aut}(G)$, and $\punit(\ofaorth(1, -))$ are trivial. Indeed, every automorphism of $\ofaorth(1, -)$ is trivial by lemma \ref{fin-et-aut} and the description of this augmented odd form $K$-algebra.

Now let us deal with the case when $G$ is a subgroup of a unitary group sheaf. The homomorphism $\punit(R, \Delta) \to \Aut(G)$ is injective by lemma \ref{density}. Hence there is a sequence of embeddings
$$G(K) / \mathbf C(G)(K) \leq \punit(R, \Delta) \leq \Aut(G).$$
It suffices to prove that all outer automorphisms of $G$ come from $\punit(R, \Delta)$. In the case $G = \glin(n, -)$ the generator of the outer automorphism group is given by $\sigma \in \punit(\ofalin(n, K))$ from lemma \ref{gl-out} if $n > 0$ (and the outer automorphism group is trivial if $n = 0$). Similarly, $\mathrm{Out}(\slin(n, -)^m) \cong (\mathbb Z / 2 \mathbb Z)^m \rtimes \mathrm S_m$ is embedded into $\punit(\ofalin(n, K^m))$ for $n \geq 3$. In the cases $G = \symp(2n, -)^m$ and $G = \sorth(2n + 1, -)^m$ the outer automorphism groups are the group of $K$-points of the discrete group scheme $\mathrm S_m$. In the case $G = \sorth(2n, -)$ the generator is given by an element of $\orth(2n, K) = \unit(R, \Delta)$ with Dickson invariant $1$ if $n > 0$ (and the outer automorphism group is trivial if $n = 0$), here $G = \sorth(8, -)$ is not an exception. For example, one may take a reflection in the orthogonal complement of the vector $e_1 - e_{-1}$. Hence the group $\mathrm{Out}(\sorth(2n, -)^m) \cong (\mathbb Z / 2 \mathbb Z)^m \rtimes \mathrm S_m$ is embedded into $\punit(\ofaorth(2n, K^m))$ for $n \geq 2$.

Now consider the case when $G$ is of a projective type. As the first part of the proof shows, $G$ is an open subgroup of finite index in $\punit\bigl((-) \otimes (R, \Delta, \mathcal D)\bigr)$. It has connected fibers, hence it is characteristic. We have a sequence of homomorphisms
$$G(K) \to \punit(R, \Delta) \to \Aut(G),$$
so it suffices to show that all outer automorphisms of $G$ come from unique cosets of $G(K)$ in $\punit(R, \Delta)$. The outer automorphisms of $G$ are the same as the outer automorphisms of the corresponding semi-simple subgroup of the unitary group scheme considered above. The additional exceptions are $\psorth(8, -)^m$ for $m \geq 1$ with the outer automorphism group $\mathrm S_3^m \rtimes \mathrm S_m$ and $\psorth(4, -)^m$ for $m \geq 2$ with the outer automorphism group $\mathrm S_{2m}$. The description of cosets follows from the first part of the proof. 
\end{proof}

From the theorem we obtain a description of twisted forms of classical groups. If $G$ is one of the group schemes from the theorem, then its twisted forms in the fppf topology over $K$ are classified by the non-abelian cohomology set $\check{\mathrm H}^1_{\mathrm{fppf}}(K, \mathbf{Aut}(G))$. Similarly, twisted forms of the corresponding classical odd form algebra $(R, \Delta, D)$ are classified by $\check{\mathrm H}^1_{\mathrm{fppf}}\bigl(K, \punit\bigr((-) \otimes_K (R, \Delta, \mathcal D)\bigl)\bigr)$. By the theorem there is a canonical bojection between these sets. This means that every twisted form of $G$ is obtained from a unique twisted form of $(R, \Delta, \mathcal D)$ (as a subgroup of the unitary group or the projective unitary group). In particular, every twisted form of a classical semi-simple group scheme of adjoint type is an open subscheme of some projective unitary group scheme.

If $G$ is a proper subgroup of $\unit\bigl((-) \otimes_K (R, \Delta, \mathcal D)\bigr)$ from the theorem, then the additional equation descents to twisted forms of $(R, \Delta, \mathcal D)$. The Dickson map after descent takes values in a twisted form of $(\mathbb Z / 2 \mathbb Z)^m$, the determinant in the linear case takes values in a twisted form of the torus $\glin(1, -)^m$.

\section{Other classical groups}

In this section we study the general unitary groups and the spin groups.

Let $(M, B, q)$ be a quadratic module over an even quadratic $K$-algebra $(R, L, A)$. Its general unitary group is
$$\gunit(M, B, q) = \{g \in \Aut_R(M) \mid \text{exists } \lambda \in K^* \text{ such that } B_M(gm, gm') = \lambda B_M(m, m'), q_M(gm) = \lambda q_M(m)\}.$$

Hence we have the groups schemes $\gorth(n, -)$, $\gsymp(2n, -)$, and a funny group scheme $\gglin(n, -)$.

There are decompositions $\gorth(2n, -) \cong \orth(2n, -) \rtimes \glin(1, -)$, $\gsymp(2n, -) \cong \symp(2n, -) \rtimes \glin(1, -)$, and $\gglin(n, -) \cong \glin(n, -) \times \glin(1, -)$ for $n > 0$. Also $\gorth(2n + 1, -) \cong \sorth(2n + 1, -) \times \glin(1, -)$ for any $n$. If $n > 0$, then we may continue the Dickson map to a homomorphism $\Dickson \colon \gorth(2n, -) \to \mathbb Z / 2 \mathbb Z$ such that $\det(g) = (1 - 2 \Dickson(g)) \lambda(g)^n$, where $\lambda \colon \gorth(2n, -) \to \glin(1, -)$ is the map from the definition. Its kernel is denoted by $\gsorth(2n, -)$, it is isomorphic to $\sorth(2n, -) \rtimes \glin(1, -)$. Let $\gorth(0, K) = \gsymp(0, K) = \gsorth(0, K) = \gglin(0, K) = 1$.

Let $(R, \Delta) \subseteq (T, \Xi)$ be a pair of odd form algebras. A general unitary group $\gunit(T, \Xi; R, \Delta)$ consists of elements $g \in \unit(T, \Xi)$ such that $\up g R = R$ and $\up g \Delta = \Delta$, it is an overgroup of $\unit(R, \Delta)$. There is an obvious homomorphism $\gunit(T, \Xi; R, \Delta) \to \punit(R, \Delta)$.

We say that the pairs of augmented odd form algebras $\ofasymp(2n, K) \subseteq \ofalin(2n, K)$, $\ofaorth(2n, K) \subseteq \ofalin(2n, K)$, and $\ofalin(n, K) \subseteq \ofalin(n, K \times K)$ are classical (as well as their twisted forms), the inclusions are given by the naive construction. More explicitly, the maps $\ofasymp(2n, K) \to \ofalin(2n, K)$ and $\ofaorth(2n, K) \to \ofalin(2n, K)$ are given by $e_{ij} \mapsto e_{ij} \pm e_{\inv{-i} \inv{-j}}$, where the ring of $\ofalin(2n, K)$ is $T = \bigoplus_{0 < |i|, |j| \leq n} Ke_{\inv{\,i\,} \inv{\,j\,}} \oplus \bigoplus_{0 < |i|, |j| \leq n} Ke_{ij}$ with the involution $\inv{e_{ij}} = e_{\inv{\,j\,} \inv{\,i\,}}$ (here $\inv{\inv{\,i\,}} = i$) and the sign is taken from the formulas for the involution on $R$. It is clear that the maps $\gsymp(2n, K) \to \gunit(\ofalin(2n, K), \ofasymp(2n, K))$, $\gorth(2n, K) \to \gunit(\ofalin(2n, K), \ofaorth(2n, K))$, and $\gglin(n, K) \to \gunit(\ofalin(n, K \times K), \ofalin(n, K))$ are isomorphisms.

If $(R, \Delta) \subseteq (T, \Xi)$ is a pair of odd form algebras, then $\Aut(T, \Xi; R, \Delta)$ is the subgroup of $\punit(T, \Xi)$ stabilizing $(R, \Delta)$. In the classical cases the outer automorphism $\sigma \colon e_{ij} \mapsto e_{\inv{-i} \inv{-j}}, e_{\inv{\,i\,} \inv{\,j\,}} \mapsto e_{-i, -j}$ of $\ofalin(2n, K)$ stabilizes $\ofasymp(2n, K)$ (acting as the matrix $\sum_{i < 0} e_{ii} - \sum_{i > 0} e_{ii}$ from $\gsymp(2n, K)$ on this odd form subalgebra) and centralizes $\ofaorth(2n, K)$.

\begin{theorem}\label{gunit-aut} 
Let $G$ be either $\gsymp(2n, -)$ or $\gsorth(2n, -)$ over a commutative ring $K$. Let $(R, \Delta, D) \subseteq (T, \Xi, X)$ be the corresponding classical pair of sheaves of augmented odd form algebras. Then the map $\Aut(T, \Xi; R, \Delta) \to \Aut(G)$ is an isomorphism with the exception $G = \gsorth(2, -)$.
\end{theorem}
\begin{proof}
If $n = 0$, then all groups from the statement are trivial. Else consider the sequence $G \to \Aut(T, \Xi; R, \Delta) \to \Aut(G)$. The kernel of the first map is $\glin(1, -)$, it coincides with the kernel of the composition. The second map is injective since if $\tau$ lies in the kernel, then it centralizes both $R$ and $\Cent(T)$ by lemma \ref{density}, i.e. it is trivial. It suffices to show that all outer automorphisms of $G$ lie in the image of the second map. Similarly to $\glin(n, -)$, it turns out that $\mathrm{Out}(\gsymp(2n, -)) \cong \mathbb Z / 2 \mathbb Z$ and $\mathrm{Out}(\gorth(2n, -)) \cong \mathbb Z / 2 \mathbb Z \times \mathbb Z / 2 \mathbb Z$ for $n > 0$. In the symplectic case the outer automorphism group is generated by $\sigma$ from the above, and in the even orthogonal case the outer automorphism group is generated by $\sigma$ and a reflection from $\orth(2n, -)$ with Dickson invariant $1$.
\end{proof}

Now all twisted forms of $\gsymp(2n, -)$ and $\gsorth(2n, -)$ arise as general unitary groups (or their index $2$ subgroups) of twisted classical pairs of odd form algebras. It seems that the group schemes $\gsorth(2n + 1, -)$ do not arise as subgroups of general unitary groups of $\ofaorth(2n + 1, -) \subseteq (T, \Xi, X)$ for some nice $(T, \Xi, X)$. On the other hand, the decomposition $\gsorth(2n + 1, -) \cong \sorth(2n + 1, -) \times \glin(1, -)$ is characteristic, hence any twisted form of $\gsorth(2n + 1, -)$ is uniquely decomposed into a product of a twisted form of $\sorth(2n + 1, -)$ and a twisted form of $\glin(1, -)$, both of them are given by augmented odd form algebras.

Let $M$ be the split classical quadratic module over a commutative ring $K$ of rank $n$. Its Clifford algebra $\clif(n, K)$ is the unital associative $K$-algebra generated by $M$ with the relation $m^2 = q(m)$. It is well-known (see, for example, \cite{Knus}) that $\clif(n, K)$ is a $2$-graded algebra, its even part $\clif_0(n, K)$ is generated by $M^2$, and its odd part $\clif_1(n, K)$ contains an isomorphic image of $M$. The Clifford algebra also has an involution $m_1 \cdots m_n \mapsto m_n \cdots m_1$ for $m_i \in M$. The spin group is $\spin(n, K) = \{u \in \clif_0(n, K)^* \mid u^{-1} = \inv u, \up uM = M\}$ (it suffices to check $\up uM \leq M$ instead of $\up uM = M$), it is reductive for $n \geq 2$. There is a short exact sequence $\mu_2 \to \spin(n, K) \to \sorth(n, K)$ for all $n$, where the right arrow is the action on $M$.

Let $(R, \Delta, D) = \ofaorth(n, K)$ be the orthogonal odd form algebra, and $M$ be the split classical quadratic module of rank $n$ as before. We construct a part of the Clifford algebra using $R$. Note that $R = M \otimes M^\oppose \cong M \otimes M$ using the isomorphism $M \cong M^\oppose, m \mapsto m^\oppose$ of $K$-modules. It is easy to see that the action of $\unit(R, \Delta)$ on $R$ is given by $\up g{(m \otimes m')} = \up g m \otimes \up g {m'}$. It follows that the natural action of the symmetric group $\mathrm S_{2k}$ on $R \otimes \ldots \otimes R \cong (M \otimes M) \otimes \ldots \otimes (M \otimes M)$ is $\punit(R, \Delta)$-invariant if $n \neq 8$, because then $\unit(R, \Delta) \to \punit(R, \Delta)$ is an epimorphism of fppf sheaves. There is a linear map $\mathrm{htr} \colon \{x \in R \mid x = \inv x\} \to K$ such that $\mathrm{htr}(e_{ij} + e_{-j, -i}) = 0$ for $i \neq j$, $\mathrm{htr}(e_{i, -i}) = 0$ for $i \neq 0$, $\mathrm{htr}(e_{ii} + e_{-i, -i}) = 1$ for $i \neq 0$, and $\mathrm{htr}(e_{00}) = 1$ (if $n$ is odd). This map satisfies $\mathrm{htr}(x + \inv x) = \tr(x)$ for even $n$ and $\mathrm{htr}(x + \inv x) = \tr(\rep(x))$ for odd $n$, it is invariant under the action of $\punit(R, \Delta)$ (since $\{x \in R \mid x = \inv x\} = \{y + \inv y \mid y \in R\}$ for even $n$ and there is an embedding $\ofaorth(n, K) \to \ofaorth(n + 1, K)$ for odd $n$). The even part of the Clifford algebra $\clif_0(n, K)$ is naturally isomorphic to the unital $K$-algebra $\clif_0(\ofaorth(n, K))$ generated by the $K$-module $R$ with the relations $x = \mathrm{htr}(x)$ and $zw = \mathrm{htr}(x) y$ for $x = \inv x \in R$ and $y \in R$, where $z \otimes w = \sigma(x \otimes y)$ and $\sigma$ is the permutation $2\,3\,1\,4$.

In order to define the spin group in terms of $(R, \Delta, D)$ we need more than just $\clif_0(\ofaorth(n, K)) \cong \clif_0(n, K)$. The odd part $\clif_1(n, K)$ cannot be constructed in the same way, since the group scheme $\punit(\ofaorth(n, K))$ does not act on this module. Instead we define $\clif_1^2(\ofaorth(n, K))$ to be a $\clif_0(\ofaorth(n, K))^{\otimes 2}$-bimodule generated by a $K$-module $R$ with the relations $(x \otimes 1) y = z_1 (w_1 \otimes 1)$ and $(1 \otimes x) y = z_2 (1 \otimes w_2)$ for $x, y \in R$, where $z_1 \otimes w_1 = \sigma_1(x \otimes y)$, $z_2 \otimes w_2 = \sigma_2(x \otimes y)$, $\sigma_1$ is the permutation $1\, 3\, 4\, 2$, and $\sigma_2$ is the permutation $2\, 3\, 1\, 4$. This bimodule is isomorphic to the bimodule $\clif_1(n, K) \otimes \clif_1(n, K)$. Hence the spin group is
$$\spin(n, K) = \{g \in \clif_0(R, \Delta) \mid g^{-1} = \inv g, (g \otimes 1) R (g^{-1} \otimes 1) = R\}.$$

\begin{theorem}\label{spin-aut}
Let $K$ be a commutative ring, $n, m \geq 0$ be integers. Then the homomorphism of group $K$-schemes $\punit(\ofaorth(n, (-)^m)) \to \Aut(\spin(n, -)^m)$ is an isomorphism with the exceptions $n \leq 2$ and $m > 1$; $n = 8$ and $m > 0$.
\end{theorem}
\begin{proof}
If $n = 0$ or $n = 1$, then $\spin(n, -) \cong \mu_2$ has trivial automorphism group, as well as $\ofaorth(n, -)$. In the case $n = 2$ the group scheme $\spin(n, -)$ is isomorphic to $\glin(1, -)$ and its automorphism group scheme $\mu_2$ is generated by a reflection from $\orth(2, -)$ with Dickson invariant $1$. Finally, if $n \geq 3$, then $\spin(n, -)^m$ is semi-simple and $\Aut(\spin(n, -)^m) \cong \Aut(\psorth(n, -)^m)$, hence the result follows from theorem \ref{ofa-aut}.
\end{proof}

In the end we briefly discuss the exceptions for our results.

The group schemes $\psorth(8, -)$ and $\spin(8, -)$ have outer automorphism group schemes $\mathrm S_3$ instead of $\mathbb Z / 2 \mathbb Z$, hence $\Aut(\ofaorth(8, -))$ is a non-normal subgroup of index $3$ in their automorphism group schemes. For these group schemes our constructions using augmented odd form algebras give only a part of all possible forms.

The sequence $\psorth(2, -) \to \punit(\ofaorth(2, -)) \to \Aut(\psorth(2, -))$ is split short exact, i.e. $\punit(\ofaorth(2, -)) \cong \glin(1, -) \rtimes \mu_2$, and the natural map $\Aut(\sorth(2, -)) \to \Aut(\psorth(2, -))$ is an isomorphism. The forms $\psorth(2, -)$ and $\sorth(2, -)$ may be classified using the isomorphisms $\psorth(2, -) \cong \sorth(2, -) \cong \glin(1, -)$.

The group schemes $\slin(1, -)$ and $\pglin(1, -)$, as well as their automorphism groups, are trivial, but $\punit(\ofalin(1, -)) \cong \mu_2$. The sequence $\pglin(2, -) \to \punit(\ofalin(2, -)) \to \Aut(\pglin(2, -))$ is isomorphic to $\pglin(2, -) \to \pglin(2, -) \times \mathbb Z / 2 \mathbb Z \to \pglin(2, -)$, also the map $\Aut(\slin(2, -)) \to \Aut(\pglin(2, -))$ is an isomorphism. The forms of $\slin(2, -)^m$ and $\pglin(2, -)^m$ may be classified using the isomorphisms $\symp(2, -) \cong \slin(2, -)$ and $\psymp(2, -) \cong \pglin(2, -)$.

We do not try to describe all twisted forms of $\glin(n, -)^m$, $\gsymp(2n, -)^m$, and $\gorth(2n, -)^m$ for $m \geq 2$ and $n \geq 1$, since their automorphism groups are not affine (they have infinite discrete group schemes of outer automorphisms). We do not consider $\gglin(n, -)$, $n \geq 1$ and $\gsorth(2, -)$ by the same reason. On the other hand, every twisted form of $\pglin(n, -)^m$, $\psymp(2n, -)^m$, or $\psorth(2n, -)^m$ determines a canonical twisted form of $\glin(n, -)^m$, $\gsymp(2n, -)^m$, or $\gorth(2n, -)^m$ using odd form algebras.

\section{Classical isotropic reductive groups}

In this section we usually assume that $K$ is a Noetherian commutative ring with connected spectrum and $G$ is a reductive group scheme over $K$ of adjoint type constructed by a twisted form $(R, \Delta, \mathcal D)$ of $\ofalin(n, K)^m$ for $n \geq 3$, $\ofaorth(2n + 1, K)^m$ for $n \geq 1$, $\ofasymp(n, K)^m$ for $n \geq 1$, or $\ofaorth(2n, K)^m$ for $n \geq 3$. Moreover, we assume that $G$ is nontrivial and indecomposable into direct product.

Let $\mathcal T(G)$ be the family of all split tori in $G$ ordered by inclusion. We denote by $\mathcal H(R, \Delta)$ the family of all orthogonal hyperbolic families $H = (\eta_1, \ldots, \eta_n)$ in $(R, \Delta)$ with the additional property $R = R e_{|i|} R$ considered up to permutations of hyperbolic pairs and up to changing $\eta_i = (e_{-i}, e_i, q_{-i}, q_i)$ to $-\eta_i = (e_i, e_{-i}, q_i, q_{-i})$. The partial order on $\mathcal H(R, \Delta)$ is the following: an orthogonal hyperbolic family $H$ is called larger than $H'$ if $H'$ may be obtained from $H$ by deleting hyperbolic pairs and adding them, i.e. replacing distinct $\eta_i = (e_{-i}, e_i, q_{-i}, q_i)$ and $\eta_j = (e_{-j}, e_j, q_{-j}, q_j)$ by $\eta_i \oplus \eta_j = (e_{-i} + e_{-j}, e_i + e_j, q_{-i} \dotplus q_{-j}, q_i \dotplus q_j)$. For example, if $(R, \Delta, \mathcal D)$ is split, then the standard maximal torus is in $\mathcal T(G)$ and the standard orthogonal hyperbolic family is in $\mathcal H(R, \Delta)$.

\begin{lemma}\label{hyp-pairs}
Let $\eta$ be a hyperbolic pair in one of the augmented odd form $K$-algebras $\ofalin(n, K)$, $\ofaorth(2n + 1, K)$, $\ofasymp(2n, K)$, or $\ofaorth(2n, K)$ for $n \geq 0$. The either $\eta$ is conjugate to one of the standard hyperbolic pairs $\eta_i$ by an element of the projective unitary group or $\eta$ is decomposable fppf locally.
\end{lemma}
\begin{proof}
Let $\eta = (e_-, e_+, q_-, q_+)$ and $(R, \Delta, \mathcal D)$ be the augmented odd form $K$-algebra. We claim that $e_+ R e_+$ is an Azumaya algebra over $K$ (or a product of two Azumaya algebras over $K$ in the linear case). Indeed, this is obvious unless $(R, \Delta, \mathcal D) = \ofaorth(2n + 1, K)$ for some $n \geq 0$. In this case let us show that
$$\rep \colon R \to \mat(2n + 1, K), e_{ij} \mapsto e_{ij} \text{ for } j \neq 0, e_{i0} \mapsto 2e_{i0}$$
maps $e_+ R e_+$ bijectively to $\rep(e_+) R \rep(e_+)$. Indeed, $R$ is a left $\mat(2n + 1, K)$-module under the multiplication $e_{ij} \cdot e_{jk} = e_{ik}$ and $e_{ij} \cdot e_{kl} = 0$ for $j \neq k$. Moreover, $\rep(a \cdot x) = a \rep(x)$, $\rep(x) y = xy$ and $a \cdot xy = (a \cdot x)y$ for all $a \in \mat(2n + 1, K)$ and $x, y \in R$. Hence if $e \in R$ is an idempotent, then $\rep \colon eRe \to \rep(e) \mat(2n + 1, K) \rep(e)$ is a bijection with the inverse map $a \mapsto a \cdot e$.

From now on we may assume that $\eta$ is indecomposable fppf locally, i.e. $e_+ R e_+ \cong K$ as a $K$-algebra. The corresponding dilations gives an embedding from the split torus $T$ of rank $1$ to $\unit((-) \otimes_K (R, \Delta, \mathcal D))$. Since every torus in a reductive group scheme lies in a maximal one and all maximal tori are conjugate fppf locally by \cite[remark 3.1.5, theorem 3.2.6]{Conrad} or \cite[Exp. XIX, theorem 2.5, 2.8]{SGA3}, we may assume that $T$ is given by a nonzero coroot $(k_1, \ldots, k_n) \in \mathbb Z^n$ of the standard maximal torus. By lemma \ref{density}, $T$ generates a subsheaf of odd form algebras isomorphic to $\ofalin(1, -)$. On the other hand, by lemma \ref{density-torus} such subsheaf is strictly larger unless all nonzero $k_i$ are equal. But if there are $m > 1$ nonzero $k_i$, then $\eta$ obviously decomposes into $m$ smaller hyperbolic pairs. Hence $\eta$ is either a standard hyperbolic pair, or it is opposite to a standard one.
\end{proof}

\begin{prop}\label{parabolic}
Let $(R, \Delta, \mathcal D)$ be a classical odd form $K$-algebra with an orthogonal hyperbolic family $H = (\eta_1, \ldots, \eta_n)$. Then the group subpresheaf $P \leq \unit((-) \otimes_K (R, \Delta, \mathcal D))$ generated by $D_i$ for $0 \leq i \leq n$, $T_{ij}$ for $i < j$, $T_j$ for $0 < j$ is a parabolic group subscheme of its fiberwise connected component.
\end{prop}
\begin{proof}
It is easy to see using the commutativity relations \cite[lemma 3]{OFADefVor}, that $P$ is a group sheaf in the fppf topology. Working fppf locally, we may assume that all the idempotents of $H$ are nonzero and there is an orthogonal hyperbolic family $H' \geq H$ such that all hyperbolic pairs from $H'$ are indecomposable fppf locally. By lemma \ref{hyp-pairs}, we may further assume that $(R, \Delta, \mathcal D)$ splits and $H'$ is the standard orthogonal hyperbolic family. In this case $P$ as a group subsheaf is generated by the maximal torus and root subgroups with the roots from a parabolic subset of the root system, i.e. $P$ is parabolic.
\end{proof}

Take an orthogonal hyperbolic family $H \in \mathcal H(R, \Delta)$ of rank $n$. For every hyperbolic pair $\eta_i \in H$ there is an embedding $K^* \to \unit(R, \Delta), k \mapsto D_i(k)$. This embedding is functorial on $K$, hence we get a homorphism $\mathbb G_{\mathrm m}^n \to G$. The image of this homomorphism is a split torus of rank $n$ or $n - 1$, the second case may only occur if $H$ contains a hyperbolic pair $\eta_i = (e_{-i}, e_i, q_{-i}, q_i)$ such that $R e_i R \neq R$. In this case $(R, \Delta, \mathcal D)$ is of linear type, the ring $R e_i R$ is an Azumaya algebra over its center, and $G$ is its automorphism group scheme over the center. Hence we have a monotone map $\mathcal H(R, \Delta) \to \mathcal T(G)$.

\begin{lemma}\label{hyp-descent}
There is a monotone map $\mathcal T(G) \to \mathcal H(R, \Delta)$ such that if $T \in \mathcal T(G)$ has rank $n$, then its image $H$ is an orthogonal hyperbolic family of rank at least $n$. If $H$ contains a hyperbolic pair $\eta = (e_-, e_+, q_-, q_+)$ such that $R \neq R e_- R$, then the rank of $H$ is at least $n + 1$.
\end{lemma}
\begin{proof}
Fix $T \in \mathcal T(G)$. We construct an orthogonal hyperbolic family fppf-locally in a canonical way, so it descents to an orthogonal hyperbolic family in $(R, \Delta)$. Since every fppf cover of $\Spec(K)$ has a refinement consisting of connected affine schemes, we may assume that already for our ring $K$ the augmented odd form algebra $(R, \Delta, \mathcal D)$ splits and $T$ is a subgroup of the standard maximal torus. The group scheme $G$ now may be decomposable, but it has no decompositions invariant under all automorphisms centralizing $T$. The scheme centralizer of $T$ in $\mathbf{Aut}(G)$ is an extension of a finite subgroup $F$ of $\mathrm{Out}(G)$ considered as a constant group scheme by the reductive group scheme $\mathbf C_G(T)$, see \cite[remark 3.1.5]{Conrad} or \cite[Exp. XIX, 2.8]{SGA3}. The maximal torus of $G$ is also a maximal torus in $\mathbf C_G(T)$. Since all maximal tori of $\mathbf C_G(T)$ are conjugate fppf locally by \cite[theorem 3.2.6]{Conrad} or \cite[Exp. XIX, theorem 2.5]{SGA3}, we may assume that $F$ has a set of representatives $\widetilde F \subseteq \Aut(G)$ normalizing the maximal torus. It acts transitively on the indecomposable factors of $(R, \Delta, \mathcal D)$, otherwise $G$ would be decomposable.

Take an indecomposable factor $(R', \Delta', \mathcal D')$ of $(R, \Delta, \mathcal D)$. The projection $T'$ of $T$ to $\punit(R', \Delta')$ has rank $n$, otherwise $(R, \Delta, \mathcal D)$ has an invariant decomposition into a direct product. Let also $\widetilde T' \leq \unit(R', \Delta')$ be the preimage of $T'$, its rank is $n + 1$ in the linear case and $n$ otherwise.

Let $H$ be the standard orthogonal hyperbolic family in $(R', \Delta')$ and $\mathrm e_1, \ldots, \mathrm e_N$ be the basic weights of the maximal torus of $\unit((-) \otimes_K (R', \Delta', \mathcal D'))$ (more precisely, of its largest reductive subgroup). Suppose at the moment that $(R, \Delta, \mathcal D)$ is not of the linear type. Now delete from $H$ the hyperbolic pairs $\eta_i$ such that $\mathrm e_i$ vanishes on $\widetilde{T'}$. Consider the following equivalence relation on the remaining hyperbolic pairs from $H$ and their opposites: each such hyperbolic pair is equivalent to itself and if $\mathrm e_i \pm \mathrm e_j$ vanishes on $\widetilde{T'}$, then $\eta_i \sim \mp \eta_j$ and $-\eta_i \sim \pm \eta_j$. We construct an orthogonal hyperbolic family $H'$ by considering sums of hyperbolic pairs from each equivalence class. Such $H'$ is not yet an orthogonal hyperbolic family, because it is closed under taking the opposite hyperbolic pairs. So we take only one hyperbolic pair from every two opposite hyperbolic pairs.

If $(R, \Delta, \mathcal D)$ is of linear type, then we use a slightly different construction of $H'$. Consider the equivalence relation on $H$, where $\eta_i \sim \eta_j$ if and only if $\mathrm e_i - \mathrm e_j$ vanishes on $\widetilde{T'}$. The sums of equivalent hyperbolic pairs form an orthogonal hyperbolic family $H''$. If the stabilizer of $(R', \Delta', \mathcal D')$ in $F$ is trivial, then we take $H' = H''$. Else hyperbolic pairs of $H'$ are the sums $\eta \oplus -\eta'$ for distinct hyperbolic pairs $\eta, \eta' \in H''$ such that $f(\eta) = -\eta'$. As in the non-linear case, we modify such $H'$ by taking only one hyperbolic pair from every two opposite hyperbolic pairs.

We claim that $H'$ is invariant under the action of $\Cent_G(T)$. It is easy to see that the standard maximal torus of $G$ is a maximal torus of $\mathbf{C}_G(T)$, and the roots of $\mathbf{C}_G(T)$ are the roots of $G$ vanishing on $T$. By construction, the maximal torus and the root subgroups of $\mathbf{C}_G(T)$ centralize $H'$. Hence by fppf descent all elements of the abstract group $\Cent_G(T) = \mathbf{C}_G(T)(K)$ centralize $H'$.

Suppose that there is $f \in \widetilde F$ inducing the non-trivial automorphism of $(R', \Delta', \mathcal D')$. Since $f$ normalizes the maximal torus, it permutes the hyperbolic pairs from $H$ and their opposites. In the non-linear case if $f(\eta_i) = \eta_{-i}$ for some $i$, then $f(\mathrm e_i) = -\mathrm e_i$ and $\mathrm e_i$ vanishes on $\widetilde{T'}$. Similarly, if $f(\eta_i) = \pm \eta_j$ for $i \neq j$, then $f(\mathrm e_i) = \pm \mathrm e_j$ and $\mathrm e_i \mp \mathrm e_j$ vanishes on $\widetilde{T'}$. Hence $f$ trivially acts on $H'$. In the linear case $f$ also trivially acts on $H'$ directly by construction.

Now we have an orthogonal hyperbolic family $H'$ for each indecomposable component $(R', \Delta', \mathcal D')$. Since $\widetilde F$ acts transitively on the components, it induce unique bijections between these orthogonal hyperbolic families. Taking sums of hyperbolic pairs in such families from each orbit under the action of $\widetilde F$, we get an invarint orthogonal hyperbolic family, as required.

If $(R, \Delta, \mathcal D)$ is not of linear type, then $\widetilde{T'}$ has rank $n$ and it is easy to see that the rank of $H'$ is at least $n$. In the linear case $\widetilde{T'}$ has rank $n + 1$, so if the stabilizer of $(R', \Delta', \mathcal D')$ in $F$ is trivial, then the rank of $H'$ is at least $n + 1$. In the remaining case the stabilizer is generated by $f \in \widetilde F$ with trivial action on $T'$ and the non-trivial action on the one-dimensional split torus $\Ker(\widetilde T' \to T')$. It follows that the rank of $H'$ is at least $n$.
\end{proof}

Let us say that $(R, \Delta, \mathcal D)$ is a locally classical odd form $K$-algebra, if there is a decomposition $K = \prod_i K_i$ into a finite direct product of rings such that $(R, \Delta, \mathcal D)$ is a finite direct product of classical odd form algebras over $K_i$. For example, $\ofalin(1, K) \times \ofasymp(4, K)$ is a locally classical odd form $K$-algebra. Recall from \cite{Stavrova}, that the isotropic rank of a reductive group scheme $G$ over any commutative ring $K$ is at least $n$ if every semisimple normal subgroup scheme of $G$ contains a split torus of rank $n$.

\begin{theorem}\label{isotropic-classical}
Let $K$ be a commutative ring, $n \geq 3$, and $G$ be a semisimple group scheme over $K$ of adjoint type. Then $G$ has isotropic rank at least $n$ if and only if it is the maximal reductive subgroup of the projective unitary group scheme associated with some locally classical odd form $K$-alegbra $(R, \Delta, \mathcal D)$ and this odd form $K$-algebra contains an orthogonal hyperbolic family $\eta_1, \eta_2, \ldots$ of rank $n + 1$ with the additional property $R = R e_{|j|} R$ or of rank $n$ with the stronger property $R = R e_j R$.
\end{theorem}
\begin{proof}
By absolute Noetherian reduction we may assume that $K$ is Noetherian, and by decomposition into a product of smaller rings we may assume that $\Spec(K)$ is connected. Similarly, we may assume that $G$ is indecomposable.

Let $G$ be a twisted form of $\psorth(8, -)^m$ for $m \geq 1$ with a split torus $T \leq G$ of rank at least $3$. Locally in the fppf topology $G$ is a subgroup of the projective unitary group scheme of $\ofaorth(8, -)^m$ and the projection of $T$ to each factor $\psorth(8, -)$ still has rank at least $3$. Any lift to $\Aut(\psorth(8, -))$ of the triality cannot centralize the projection of $T$, since it acts nontrivially on a maximal torus containing $T$, but the cyclic group $\mathbb Z / 3 \mathbb Z$ has no nontrivial one-dimensional representations over $\mathbb Q$. If there is a nontrivial outer automorphism of $\psorth(8, -)$ with a lift to $\Aut(\psorth(8, -))$ centralizing $T$, then without loss of generality it is the outer automorphism of $\ofaorth(8, -)$. Hence $G$ may be obtained from a twisted form of $\ofaorth(8, -)^m$.

The rest follows from lemma \ref{hyp-descent} and the construction before it. Note that the assumption $n \geq 3$ implies that we do not have to consider twisted forms of $\ofalin(1, K)^m$, $\ofalin(2, K)^m$, $\ofaorth(1, K)^m$, $\ofaorth(2, K)^m$, and $\ofaorth(4, K)^m$.
\end{proof}

If $G$ is a reductive group scheme over $K$ with the isotropic rank $n \geq 2$, then the main result of \cite{PetrovStavrova} says that there is a well-defined elementary subgroup of $G(K)$. By proposition \ref{parabolic} the elementary subgroup may be defined in terms of unitary transvections if $G$ is constructed by a locally classical odd form $K$-algebra $(R, \Delta, \mathcal D)$) with an orthogonal hyperbolic family of rank at least $1$ (such that $R e_{|i|} R = R$ for each $i$). By theorem \ref{isotropic-classical}, such $(R, \Delta, \mathcal D)$ always exists if $n \geq 3$ and $G$ is of adjoint type.

\bibliographystyle{plain}
\bibliography{references}

\end{document}